\documentclass[11pt]{amsart} 

\usepackage{amsmath} 
\usepackage{amssymb}
\usepackage{amsxtra}
\usepackage{mathtools} 
\usepackage{epic,eepic} 
\usepackage[hypertex]{hyperref}
\usepackage[margin=3cm]{geometry} 
\usepackage{booktabs}
\usepackage{amsrefs}
\newtheorem{theorem}{Theorem}[section]
\newtheorem{lemma}[theorem]{Lemma}
\newtheorem{proposition}[theorem]{Proposition}
\newtheorem{corollary}[theorem]{Corollary}

\theoremstyle{definition} 
\newtheorem{definition}[theorem]{Definition}
\newtheorem{definition-proposition}[theorem]{Definition-Proposition} 
\newtheorem{remark}[theorem]{Remark}
\newtheorem{example}[theorem]{Example}
\newtheorem{notation}[theorem]{Notation}

\newcommand{\LC}{\scriptscriptstyle \text{LC}} 
\newcommand{\CR}{\scriptscriptstyle \text{CR}}
\newcommand{\pss}{\big[\!\big[s^0,\ldots,s^N\big]\!\big]}

\newcommand{\abs}{\mathrm{abs}}
\newcommand{\cA}{\mathcal{A}}
\newcommand{\tcA}{\widetilde{\cA}} 
\newcommand{\cAabs}{\mathcal{A}^{\text{\rm \tiny abs}}}
\newcommand{\cAabsformal}
{\mathcal{A}^{\text{\rm \tiny abs}}_{\text{\rm \tiny formal}}}
\newcommand{\cAabsan}
{\mathcal{A}^{\text{\rm \tiny abs}}_{\text{\rm \tiny an}}}
\newcommand{\A}{\mathbb A} 
\newcommand{\ba}{\mathbf{a}} 
\newcommand{\Aut}{\operatorname{Aut}}

\newcommand{\bb}{\mathbf{b}}
 
\newcommand{\bc}{\mathbf{c}} 
\newcommand{\C}{\mathbb C}
\newcommand{\Cstar}{\C^{\times}}

\newcommand{\diag}{\operatorname{diag}} 

\newcommand{\ev}{\operatorname{ev}}
\newcommand{\End}{\operatorname{End}}
\newcommand{\cE}{\mathcal{E}} 

\newcommand{\cF}{\mathcal{F}} 
\newcommand{\tcF}{\widetilde{\cF}}
\DeclareMathOperator{\FractionField}{Frac}
\newcommand{\bbf}{\mathbf{f}} 
\newcommand{\Fock}{\mathfrak{Fock}}
\newcommand{\FJRW}{\operatorname{FJRW}} 
\newcommand{\sdet}{\operatorname{sdet}} 

\newcommand{\bg}{\mathbf{g}} 

\newcommand{\Hom}{\operatorname{Hom}}
\newcommand{\cH}{\mathcal{H}} 

\DeclareMathOperator{\Id}{Id}
\newcommand{\iu}{\mathtt{i}}

\newcommand{\cM}{\mathcal{M}} 
\newcommand{\cMbar}{\ov{\cM}}
\newcommand{\bm}{\mathbf{m}}

\DeclareMathOperator{\NE}{NE}

\newcommand{\cO}{\mathcal{O}}

\newcommand{\Proj}{\mathbb P}
\newcommand{\bp}{\mathbf{p}}
\newcommand{\tbp}{\tilde{\bp}}

\newcommand{\bq}{\mathbf{q}}
\newcommand{\tbq}{\tilde{\bq}} 
\newcommand{\Q}{\mathbb Q}

\newcommand{\Res}{\operatorname{Res}}

\newcommand{\Spf}{\operatorname{Spf}}   

\newcommand{\bt}{\mathbf{t}}
\newcommand{\cT}{\mathcal{T}}

\newcommand{\bv}{\mathbf{v}}

\newcommand{\bx}{\mathbf{x}} 

\newcommand{\by}{\mathbf{y}}

\newcommand{\Z}{\mathbb Z}
\newcommand{\cZ}{\mathcal{Z}} 

\newcommand{\unit}{\boldsymbol{1}} 
\newcommand{\cS}{\Upsilon} 

\newcommand{\ov}{\overline}
\def\pair#1#2{\langle #1,#2\rangle}
\def\Pair#1#2{\left\langle #1,#2\right\rangle}
\def\parfrac#1#2{\frac{\partial{#1}}{\partial #2}}
\def\corr#1{\left\langle #1 \right\rangle} 
\def\leftsub#1{{}_{#1}}

\title[On the Convergence of Gromov--Witten Potentials]{On the
  Convergence of Gromov--Witten Potentials and Givental's Formula}
\author{Tom Coates}
\email{t.coates@imperial.ac.uk}
\author{Hiroshi Iritani}
\email{iritani@math.kyoto-u.ac.jp} 

\begin{document} 
\begin{abstract} 
  Let $X$ be a smooth projective variety.  The Gromov--Witten
  potentials of $X$ are generating functions for the Gromov--Witten
  invariants of $X$: they are formal power series, sometimes in
  infinitely many variables, with Taylor coefficients given by
  Gromov--Witten invariants of $X$.  It is natural to ask whether
  these formal power series converge.  In this paper we describe and
  analyze various notions of convergence for Gromov--Witten
  potentials.  Using results of Givental and Teleman, we show that if
  the quantum cohomology of $X$ is analytic and generically semisimple
  then the genus-$g$ Gromov--Witten potential of $X$ converges for all
  $g$.  We deduce convergence results for the all-genus Gromov--Witten
  potentials of compact toric varieties, complete flag varieties, and
  certain non-compact toric varieties.
\end{abstract} 

\maketitle 

\section{Introduction} 

Let $X$ be a smooth projective variety.  The total descendant
potential of $X$ is a generating function for the Gromov--Witten
invariants of $X$.  It is a formal power series $\cZ_X$ in $\hbar$,
$\hbar^{-1}$, and infinitely-many variables $t^\alpha_k$, $0 \leq
\alpha \leq N$, $0 \leq k < \infty$, with Taylor coefficients given by
Gromov--Witten invariants of $X$.  Here $t_0, t_1, t_2, \ldots$ is an
infinite sequence of cohomology classes on $X$, $t_k = t^0_k \phi_0 +
\cdots + t^N_k \phi_N$ is the expansion of $t_k$ in terms of a basis
$\{\phi_\alpha\}$ for $H^\bullet(X)$, and:
\[
\cZ_X = \exp \left( \sum_{g \geq 0} \hbar^{g-1} \cF^g_X \right)
\]
where $\cF^g_X$ is a generating function for genus-$g$ Gromov--Witten
invariants.  It is known that $\cZ_X$ does not converge\footnote{$\cZ_X$
  should be regarded as an asymptotic expansion in $\hbar$.} as a
series in $\hbar$ and $\hbar^{-1}$, but it is natural to ask whether
the formal power series $\cF^g_X$ converge.  This question is
particularly relevant in light of work by Ruan and his
collaborators on Gromov--Witten theory and birational geometry.  If $X
\dashrightarrow Y$ is a crepant birational map between smooth
projective varieties (or orbifolds) then, very roughly speaking, the
total descendant potentials $\cZ_X$ and $\cZ_Y$ are conjectured to be
related by analytic continuation in the parameters $t_i^\alpha$.
Implicit here, then, is the conjecture that the power series defining
$\cF^g_X$ and $\cF^g_Y$ converge.

There are several different notions of convergence for a power series
in infinitely-many variables.  We say that the total descendant
potential $\cZ_X$ is NF-convergent (see
Definition~\ref{def:descendant-converge} below) if each genus-$g$
descendant potential $\cF^g_X$ converges on an infinite-dimensional
polydisc of the form shown in equation~\ref{eq:convergence-domain-cFg}
below.  This implies that each $\cF^g$ defines a holomorphic function
on a neighbourhood of zero in an appropriate nuclear Fr\'echet space:
see Remark~\ref{rem:holomorphy}.  The main result of this paper
(Theorem~\ref{thm:maintheorem} below) is that if $X$ is a projective
variety such that the quantum cohomology of $X$ is analytic and
generically semisimple, then the total descendant potential $\cZ_X$ is
NF-convergent.

The quantum cohomology of $X$ is a family of algebra structures on
$H^\bullet(X)$ parametrized by a point $t \in H^\bullet(X)$.  The
structure constants of the quantum cohomology algebra are formal power
series in $t^\alpha$, $0 \leq \alpha \leq N$, where $t= t^0 \phi_0 +
\cdots + t^N \phi_N$ is the expansion of $t$ with respect to a basis
$\{\phi_\alpha\}$ for $H^\bullet(X)$, with Taylor coefficients given
by genus-zero Gromov--Witten invariants of $X$: see
\S\ref{sec:convergence}.  We consider three conditions on the
Gromov--Witten invariants of $X$:
\begin{description}
\item[Formal Semisimplicity] (see equation
  \ref{eq:formalsemisimplicity}), which roughly speaking states that
  the quantum cohomology algebra of $X$ is semisimple at the generic
  point of a formal neighbourhood of the large-radius limit point;
\item[Genus-Zero Convergence] (see equation
  \ref{eq:genuszeroconvergence}), which roughly speaking states that
  the power series defining the quantum cohomology algebra converge to
  give analytic functions of $t^0,\ldots,t^N$; and
\item[Analytic Semisimplicity] (see equation
  \ref{eq:analyticsemisimplicity}) which asserts that the resulting
  analytic family of algebras is semisimple for generic $t \in
  H^\bullet(X)$.
\end{description}
Formal Semisimplicity and Genus-Zero Convergence together imply
Analytic Semisimplicity, and Genus-Zero Convergence and Analytic
Semisimplicity together imply Formal Semisimplicity.

\begin{theorem}
 \label{thm:maintheorem}
 Let $X$ be a smooth projective variety that satisfies Formal
 Semisimplicity, Genus-Zero Convergence, and Analytic Semisimplicity.
 The total descendant potential $\cZ_X$ is NF-convergent in the sense
 of Definition~\ref{def:descendant-converge}..
\end{theorem}

Theorem~\ref{thm:maintheorem} is proved in
Section~\ref{sec:statements} below.  It has the following immediate
consequences.

\begin{corollary}
  \label{cor:toricflag}
  Let $X$ be a compact toric variety or a complete flag variety.  The
  total descendant potential $\cZ_X$ is NF-convergent in the sense of
  Definition~\ref{def:descendant-converge}.
\end{corollary}

\begin{proof}
  By Theorem~\ref{thm:maintheorem}, it suffices to show that $X$
  satisfies Genus-Zero Convergence and Analytic Semisimplicity.  If
  $X$ is a compact toric variety then this follows from mirror
  symmetry \citelist{\cite{Givental:toric} \cite{Hori--Vafa}
    \cite{Iritani:convergence}}.  If $X$ is a complete flag variety
  then this follows from mirror symmetry
  \citelist{\cite{Givental:flag}\cite{Kim}}, reconstruction theorems
  for logarithmic Frobenius manifolds
  \citelist{\cite{Reichelt}\cite{Iritani:convergence}*{Proposition~5.8}
  }, and the work of Kostant \cite{Kostant}.
\end{proof}

Theorem~\ref{thm:maintheorem} also implies the NF-convergence of the
total descendant potential $\cZ_X$ when $X$ is the total space of a
direct sum of negative line bundles over a compact toric variety.
This includes the case where $X = K_Y$ is the total space of the
canonical line bundle over a compact Fano toric variety $Y$.

\begin{corollary}
  \label{cor:local}
  Let $Y$ be a compact toric variety and let $X$ be the total space of a direct
  sum $E = \bigoplus_{j=1}^{j=r} E_j$ of line bundles $E_j$ over $Y$
  such that $c_1(E_j) \cdot d < 0$ whenever $d$ is the degree of a
  holomorphic curve in $Y$.  The total descendant potential
  $\cZ_X$ is NF-convergent in the sense of
  Definition~\ref{def:descendant-converge}.
\end{corollary}

Corollary~\ref{cor:local} is proved in Section~\ref{sec:local} below.  

\medskip

We deduce Theorem~\ref{thm:maintheorem} from a more fundamental
result, Theorem~\ref{thm:maintheoremancestors} below, concerning the
convergence of the total ancestor potential $\cA_X$.  The total
ancestor potential is a generating function for ancestor
Gromov--Witten invariants (see equations
\ref{eq:mixedcorrelator}--\ref{eq:ancestorpotential}).  We say that
the total ancestor potential $\cA_X$ is NF-convergent if it is
convergent on an infinite-dimensional polydisc as before (see
equation~\ref{eq:formofpolydisc}).  We consider also a stronger notion
of convergence for $\cA_X$ (see
Definition~\ref{def:ancestor-converge}), requiring that in terms of
the dilaton-shifted co-ordinates introduced in \S\ref{sec:dilaton}, we
have:
\begin{align*}
  & \cA_X = \exp \left( \sum_{g=0}^\infty \hbar^{g-1} \bar{\cF}^{g}_t
  \right) \\
\intertext{where:}
  & \bar{\cF}^{g}_t = 
\sum_{n: 3g-3+n \ge 0}  
  \frac{1}{n!}
  \sum_{\substack{ 
      I : I = (i_1,\ldots,i_n) \\
      \text{$i_j \neq 1$ for all $j$} \\ 
      i_1 + \cdots + i_n \leq 3 g -3 + n}}
  \sum_{A=(\alpha_1,\dots,\alpha_n)}  
 C^{(g)}_{I,A}(t,q_1) \,
  q_{i_1}^{\alpha_1}\cdots q_{i_n}^{\alpha_n}
\end{align*}
for some analytic functions $C^{(g)}_{I,A}(t,q_1)$ of $(t,q_1)$
that are rational in $q_1$ unless $(g,n)=(1,0)$ 
(see \eqref{eq:coeff-rationality}). 
Convergence in this sense implies that
the genus-$g$ ancestor potential $\bar{\cF}^{g}_t$ is a formal power
series in $q_0^\alpha$ with coefficients that depend polynomially on
$q_i^\alpha$, $i > 1$, and holomorphically on $t$ and $q_1^\alpha$;
furthermore Givental's tameness condition \cite{Givental:An} holds.

\begin{theorem}
  \label{thm:maintheoremancestors}
  Let $X$ be a smooth projective variety that satisfies Formal
  Semisimplicity, Genus-Zero Convergence, and Analytic Semisimplicity.
  The total ancestor potential $\cA_X$ is NF-convergent in the sense
  of Definition~\ref{def:Frechet-convergence-ancestor}, and is
  convergent in the sense of Definition~\ref{def:ancestor-converge}.
\end{theorem}

The rationality condition on $\cA_X$ and the definition of the
ancestor Fock space in which $\cA_X$ lies were developed as part of a
joint project with Hsian-Hua Tseng. 
We would like to thank him for allowing us to present 
the Fock space formulation in this paper. 

We now discuss the work of Givental
\citelist{\cite{Givental:semisimple}\cite{Givental:quantization}} and
Teleman \cite{Teleman} on higher-genus potentials for target spaces
with semisimple quantum cohomology.  This is an essential ingredient
in the proof of Theorem~\ref{thm:maintheoremancestors}.  Motivated by
an ingenious localization computation in torus-equivariant
Gromov--Witten theory, Givental conjectured a formula which determines
higher-genus Gromov--Witten potentials in terms of genus-zero data
alone.  His formula makes sense for any semisimple Frobenius manifold.
In order to distinguish it from the geometric Gromov--Witten
potential, we call the potential associated to a Frobenius manifold
via Givental's formula the \emph{abstract potential}.

Teleman has shown that for any semisimple Cohomological Field Theory
(CohFT) satisfying a homogeneity condition and a flat vacuum
condition, the potential associated to the CohFT coincides with
Givental's abstract potential \cite{Teleman}.  Since Gromov--Witten
theory defines a CohFT satisfying the homogeneity and flat vacuum
conditions, Teleman's theorem applies to Gromov--Witten theory
whenever the genus-zero part (quantum cohomology) is semisimple.
There is a subtlety here.  Quantum cohomology is a formal family of
algebras parametrized by Novikov variables $Q_i$ and cohomology
parameters $t^0, \ldots, t^N$ as above, and its convergence is not
known in general.  At the origin $Q_i=t^j=0$, the quantum cohomology
coincides with the classical cohomology ring, and so is semisimple
only when the target $X$ is a point.  At first sight, then, it appears
that to apply Teleman's theorem we need to find a semisimple point in
the parameter space where all higher-genus Gromov--Witten potentials
converge.  (To prove this directly is beyond the reach of current
methods in all but the very simplest examples.)\phantom{.}  In fact
this is not the case: as Teleman points out in
\cite{Teleman}*{Example~1.6}, his theorem applies whenever the quantum
cohomology ``at the generic point" in the formal neighbourhood of the
origin is semisimple.  Thus Givental's abstract potential can be
defined and coincides with the geometric Gromov--Witten potential
under our assumption of Formal Semisimplicity
\eqref{eq:formalsemisimplicity}.  If in addition Genus-Zero
Convergence holds then \emph{it follows} that the higher-genus
Gromov--Witten potentials, which \emph{a priori} are only formal power
series, in fact converge to give analytic functions.

We expand upon these points in the rest of the paper.  In
\S\ref{sec:notation} we fix notation for Gromov--Witten invariants,
generating functions, and quantum cohomology.  In
\S\ref{sec:quantization} we describe Givental's quantization
formalism.  We then discuss Givental's formula in the analytic setting
(\S\ref{sec:Giventalformulaanalytic}) and in the formal setting
(\S\ref{sec:Giventalformulaformal}), and explain how Givental's
formula follows from Teleman's classification theorem
(\S\ref{sec:TelemanimpliesGivental}).  Results about the NF
convergence of ancestor and descendant potentials are stated in
\S\ref{sec:statements} and proved in \S\ref{sec:proofs}.  We conclude
with the proof of Corollary~\ref{cor:local} in \S\ref{sec:local}.

\subsection*{Acknowledgements}  
We are grateful to Hsian-Hua Tseng for very useful discussions on
Givental quantization.  The definition of Fock spaces for ancestor
potentials was originally worked out in another joint project with
him, and we thank him for allowing us to present this formulation
here.  We thank Yongbin Ruan and Yefeng Shen for giving us a strong
motivation for writing up this paper. TC thanks Konstanze Rietsch for
a useful conversation about flag varieties.  This research is
supported by TC's Royal Society University Research Fellowship, ERC
Starting Investigator Grant number~240123, the Leverhulme Trust, and
Grant-in-Aid for Scientific Research (S)~23224002 and 
Grant-in-Aid for Young Scientists (B)~22740042.  

\section{Preliminaries} 
\label{sec:notation}
Let $X$ be a smooth projective variety 
and let $H_X$ be the even part of $H^\bullet(X;\Q)$.

\subsection{Gromov--Witten Invariants}
\label{sec:GW}
Let $X_{g,n,d}$ denote the moduli space of $n$-pointed
genus-$g$ stable maps to $X$ of degree $d \in H_2(X;\Z)$.
Write:
\begin{align}
  \label{eq:correlator}
  \corr{a_1 \psi_1^{i_1},\ldots,a_n
    \psi_n^{i_n}}^{X}_{g,n,d}
  =
  \int_{[X_{g,n,d}]^{\text{vir}}}
  \prod_{k=1}^{k=n} \ev_k^\star(a_k) \cup \psi_k^{i_k}
\end{align}
where $a_1,\ldots,a_n \in H_X$; $\ev_k \colon X_{g,n,d} \to X$ is the
evaluation map at the $k$th marked point; $\psi_1,\ldots,\psi_n \in
H^2\big(X_{g,n,d};\Q\big)$ are the universal cotangent line classes;
$i_1, \ldots, i_n$ are non-negative integers; and the integral denotes
cap product with the virtual fundamental class
\citelist{\cite{Behrend--Fantechi} \cite{Li--Tian}}.  The right-hand
side of \eqref{eq:correlator} is a rational number, called a
\emph{Gromov--Witten invariant} of $X$ (if $i_k = 0$ for all $k$) or a
\emph{gravitational descendant} (if any of the $i_k$ are non-zero).

\subsection{Bases for Cohomology and Novikov Rings}
\label{sec:bases}

Fix bases $\phi_0,\ldots,\phi_N$ and $\phi^0,\ldots,\phi^N$ for $H_X$
such that:
\begin{equation}
  \label{eq:basisproperties}
  \begin{minipage}{0.9\linewidth}
    \begin{itemize}
    \item $\phi_0$ is the identity element of $H_X$
    \item $\phi_1,\ldots,\phi_r$ is a nef $\Z$-basis for 
 $H^2(X;\Z) \subset H_X$ 
    \item each $\phi_i$ is homogeneous
    \item $(\phi_i)_{i=0}^{i=N}$ and $(\phi^j)_{j=0}^{j=N}$ are dual with
      respect to the Poincar\'e pairing
    \end{itemize}
  \end{minipage}
\end{equation}
Note that $r$ is the rank of $H_2(X)$. 
Define the \emph{Novikov ring} 
$\Lambda = \Q[\![Q_1,\ldots,Q_r]\!]$ and, for $d \in H_2(X;\Z)$, write:
\[
Q^d = Q_1^{d_1} \cdots Q_r^{d_r}
\]
where $d_i = d\cdot \phi_i$. 

\subsection{Quantum Cohomology}
\label{sec:convergence}
Let $t^0,\ldots,t^N$ be the co-ordinates on $H_X$ defined by the basis
$\phi_0,\ldots,\phi_N$, so that $t \in H_X$ satisfies $t = t^0 \phi_0
+ \ldots + t^N \phi_N$.  
Define the \emph{genus-zero Gromov--Witten potential} 
$F^0_X \in \Lambda[\![t^0,\dots,t^N]\!]$ by: 
\[
F^0_X = \sum_{d \in \NE(X)} \sum_{n = 0}^\infty \frac{Q^d}{n!}
\corr{\vphantom{\big\vert}t, \ldots,t}^X_{0,n,d}
\]
where the first sum is over the set $\NE(X)$ of degrees of effective
curves in $X$.  This is a generating function for genus-zero
Gromov--Witten invariants.  The \emph{quantum product} $\ast$ is
defined in terms of the third partial derivatives of $F^0_X$:
\begin{equation}
  \label{eq:bigQC}
  \phi_\alpha \ast \phi_\beta = \sum_{\gamma=0}^{\gamma=N} 
  {\partial^3 F^0_X \over \partial t^\alpha \partial t^\beta \partial t^\gamma}
  \phi^\gamma 
\end{equation}
The product $\ast$ is bilinear over $\Lambda$, and defines a formal
family of algebras on $H_X \otimes \Lambda$ parameterized by
$t^0,\ldots,t^N$.  This is the \emph{quantum cohomology} or
\emph{big quantum cohomology} of $X$.

We have defined big quantum cohomology as a formal family of algebras,
i.e.~in terms of the ring of formal power series
$\Q[\![Q_1,\ldots,Q_r]\!][\![t^0,\ldots,t^N]\!]$.  In many cases
however, the genus-zero Gromov--Witten potential $F^0_X$ converges to
an analytic function.  By this we mean the following.  The Divisor
Equation \cite{Kontsevich--Manin}*{\S2.2.4} implies that:
\[
F^0_X \in \Q[\![t^0,Q_1 e^{t^1}, \ldots, Q_r e^{t^r}, t^{r+1},t^{r+2},\ldots,t^N]\!]
\]
and one can often show, for example by using mirror symmetry, that
$F^0_X$ is the power series expansion of an analytic function:
\[
F^0_X \in \Q\Big\{t^0,Q_1 e^{t^1}, \ldots, Q_r e^{t^r},
t^{r+1}, t^{r+2}, \ldots, t^N\Big\}
\]
We can then set $Q_1 = \cdots = Q_r = 1$, obtaining an analytic
function:
\[
F^0_X \in \Q\Big\{t^0, e^{t^1}, \ldots, e^{t^r},
t^{r+1}, t^{r+2},\ldots,t^N\Big\}
\]
of the variables $t^0,\ldots, t^N$ defined in a region:
\begin{align}
  \label{eq:LRLnbhd}
  \begin{cases}
    |t^i| < \epsilon_i & \text{$i=0$ or $r<i\leq N$} \\
    \Re t^i \ll 0 & 1 \leq i \leq r
  \end{cases}
\end{align}
We refer to the limit point
\begin{align*}
  \begin{cases}
    t^i =0  & \text{$i=0$ or $r<i\leq N$} \\
    \Re t^i \to -\infty & 1 \leq i \leq r
  \end{cases}
\end{align*}
as the \emph{large-radius limit point}.  When $F^0_X$ converges to an
analytic function in the sense just described, the quantum product
$\ast$ then defines a family of algebra structures on $H_X$ that
depends analytically on parameters $t^0,\ldots,t^N$ in the
neighbourhood \eqref{eq:LRLnbhd} of the large-radius limit point.

\begin{remark} 
  In this paper we only consider the even part of the cohomology
  group, but this is not really a restriction.
  Hertling--Manin--Teleman \cite{HMT} proved that if the quantum
  cohomology of a smooth projective variety $X$ is semisimple, then
  $X$ has no odd cohomology and is of Hodge--Tate type: $H^{p,q}(X)=0$
  for $p\neq q$.
\end{remark} 

\subsection{The Dubrovin Connection}
\label{sec:Dubrovinconnection}

Consider $H_X \otimes \Lambda$ as a scheme over $\Lambda$ and let
$\cM$ be a formal neighbourhood of the origin in $\cM$.  The
\emph{Euler vector field} $E$ on $\cM$ is:
\begin{equation}
  \label{eq:Eulerfield}
  E = {t^0 \parfrac{}{t^0}} + \sum_{i=1}^r \rho^i \parfrac{}{t^i} +
  \sum_{i=r+1}^N \big(1 - \textstyle\frac{1}{2}{\deg \phi_i} \big) 
t^i \parfrac{}{t^i}
\end{equation}
where $c_1(X) = \rho^1 \phi_1 + \cdots + \rho^r \phi_r$.  The
\emph{grading operator} $\mu \colon H_X \to H_X$ is defined by:
\[
\mu(\phi_i) = \left(\textstyle\frac{1}{2} \deg \phi_i -
  \textstyle\frac{1}{2} \dim_{\C} X\right) \phi_i
\]
Let $\pi \colon \cM \times \A^1 \to \cM$ 
denote projection to the first factor.  
The \emph{extended Dubrovin connection} 
is a meromorphic flat connection $\nabla$ 
on $\pi^\star T\cM \cong H_X \times (\cM\times \A^1)$, defined by:
\begin{align*}
& \nabla_{\parfrac{}{t_i}} = \parfrac{}{t_i} 
  - \frac{1}{z}\big(\phi_i {\ast}\big)
&& 0 \leq i \leq N \\
& \nabla_{z \parfrac{}{z}} = z \parfrac{}{z} + \frac{1}{z} 
\big({E\ast}\big) + \mu 
&& \text{where $z$ is the co-ordinate on $\A^1$.} 
\end{align*}
Together with the pairing on $T\cM$ induced by the Poincar\'e pairing,
the Dubrovin connection equips $\cM$ with the structure of a formal
Frobenius manifold with extended structure
connection \cite{Manin}.

If the genus-zero Gromov--Witten potential $F^0_X$ converges to an
analytic function, as discussed in Section~\ref{sec:convergence} above, 
then the extended Dubrovin connection with $Q_1= \cdots = Q_r =1$ 
depends analytically on $t$ in a neighbourhood \eqref{eq:LRLnbhd} 
of the large-radius limit point 
and defines an analytic Frobenius manifold 
with extended structure connection. 

\subsection{Gromov--Witten Potentials}
\label{sec:GW-potentials} 

We begin by defining the formal power series ring to which the
Gromov--Witten potentials belong.  The Novikov ring $\Lambda$ is
topologized by regarding it as the completion of the polynomial ring
$\Q[Q_1,\ldots,Q_r]$ with respect to the valuation $v$ such that
$v(Q^d) = d \cdot \omega$, where $\omega$ is a K\"ahler class on $X$.
We will need also certain related formal power series rings, shown in
Table~\ref{tab:rings}.  These are defined as the completions of
polynomial rings, shown in the second column of Table~\ref{tab:rings},
with respect to a valuation $v$ such that:
\begin{alignat*}{4} 
  & v(Q^d) = d \cdot \omega, & \quad 
  & v(t^\alpha) = 1, & \quad 
  & v(t^\alpha_i) = i+1, & \quad 
  & v(y^\beta_j) = j+1.  
\end{alignat*} 
For a ring $R$ equipped with non-negative valuation $v$, we define:
\[
R\{\hbar^{-1},\hbar]\!] = 
\Bigg\{ \sum_{n = -\infty}^{n=\infty} a_n \hbar^{n} : \text{$a_n \in
  R$, $\lim_{n \to -\infty} v(a_n) = \infty$}\Bigg\}. 
\]

\begin{table}
  \centering
  \begin{tabular}{l@{\hspace{6ex}}ll}
    \toprule
    \multicolumn{1}{c}{Completed Ring} & 
    \multicolumn{1}{c}{Underlying Polynomial Ring} \\ \midrule
    $\Lambda$ & $\Q[Q_1,\ldots,Q_r]$ \\
   $\Lambda[\![t]\!]$ & $\Q[Q_1,\ldots,Q_r]
[t^\alpha : 0 \leq \alpha \leq N]$ \\
    $\Lambda[\![\bt]\!]$ & $\Q[Q_1,\ldots,Q_r][t_i^\alpha :
    0 \leq i < \infty, 0 \leq \alpha \leq N]$ \\
    $\Lambda[\![\by]\!][\![t]\!]$ & $\Q[Q_1,\ldots,Q_r][y_j^\beta:
    0 \leq j < \infty, 0 \leq \beta \leq N][t^\alpha : 0 \leq \alpha \leq N]$\\
   \bottomrule \\
  \end{tabular}
  \caption{Formal Power Series Rings}
  \label{tab:rings}
\end{table}

Let $\bt =(t_0, t_1, t_2, \ldots)$ be an infinite sequence of elements
of $H_X$ and write $t_i = t_i^0 \phi_0 + \cdots + t_i^N \phi_N$.
Define the \emph{genus-$g$ descendant potential} $\cF^g_X \in
\Lambda[\![\bt]\!]$ by:
\begin{equation} 
\label{eq:genus_g_descendantpot}
\cF^g_X = \sum_{d \in \NE(X)} \sum_{n = 0}^{\infty} 
\sum_{i_1=0}^\infty
\cdots
\sum_{i_n=0}^\infty 
{Q^d \over n!}
\corr{t_{i_1} \psi_1^{i_1},\ldots,t_{i_n}\psi_n^{i_n}}^X_{g,n,d}. 
\end{equation} 
This is a generating function for genus-$g$ gravitational descendants.
The \emph{total descendant potential} $\cZ_X \in
\Lambda[\![\bt]\!]\{\hbar^{-1},\hbar]\!]$ is:
\begin{equation} 
\label{eq:totaldescendantpot}
\cZ_X = \exp \Bigg(\sum_{g=0}^\infty \hbar^{g-1} \cF^g_X\Bigg). 
\end{equation}
This is a generating function 
for all gravitational descendants of $X$. 

Consider now the map $p_m \colon X_{g,m+n,d} \to \cMbar_{g,m}$ that
forgets the map and the last $n$ marked points, and then stabilises
the resulting prestable curve.  Write $\psi_{m|i} \in
H^2(X_{g,n+m,d};\Q)$ for the pullback along $p_m$ of the $i$th
universal cotangent line class on $\cMbar_{g,m}$, and:
\begin{multline}
  \label{eq:mixedcorrelator}
  \corr{a_1 \bar{\psi}_1^{i_1},\ldots,a_m
    \bar{\psi}^{i_m}:b_1,\ldots,b_n}^{X}_{g,m+n,d} \\
  =
  \int_{[X_{g,m+n,d}]^{\text{vir}}}
  \prod_{k=1}^{k=m} 
  \Big(\ev_k^\star(a_k) \cup \psi_{m|k}^{i_k}\Big)
  \cdot
  \prod_{l=m+1}^{l=m+n} \ev_{l}^\star (b_{l-m})  
\end{multline}
where $a_1,\ldots,a_m \in H_X$; $b_1,\ldots,b_n \in
H_X$; and $i_1,\ldots,i_m$ are non-negative
integers.

As above, consider $t \in H_X$ with $t = t^0\phi_0 + \cdots + t^N
\phi_N$ and an infinite sequence 
$\by =(y_0, y_1, y_2, \ldots)$ of elements in $H_X$ 
with $y_i = y_i^0 \phi_0 + \cdots + y_i^N \phi_N$.  
The \emph{genus-$g$ ancestor potential} 
$\bar{\cF}^g_X \in \Lambda[\![\by]\!][\![t]\!]$ 
is defined by:
\begin{equation}
  \label{eq:ancestor}
  \bar{\cF}^g_{X} = 
  \sum_{d \in \NE(X)}
  \sum_{n=0}^\infty
  \sum_{m=0}^\infty
  \sum_{j_1=0}^\infty
  \cdots
  \sum_{j_m=0}^\infty
  {Q^d \over n!m!}
 \corr{y_{j_1} \bar{\psi}_{1}^{j_1},\ldots,y_{j_m} \bar{\psi}_{m}^{j_m} : 
    \overbrace{t,\ldots, t}^n}^X_{g,m+n,d}
\end{equation}
and the \emph{total ancestor potential} $\cA_X \in
\Lambda[\![\by]\!][\![t]\!]\{\hbar^{-1},\hbar]\!]$ is:
\begin{equation}
  \label{eq:ancestorpotential}
  \cA_{X} = \exp \Bigg(\sum_{g=0}^\infty \hbar^{g-1} \bar{\cF}^g_{X}\Bigg)
\end{equation}
We will often want to emphasize the dependence of the ancestor
potentials on the variable $t$, writing $\bar{\cF}^g_t$ for
$\bar{\cF}^g_X$ and $\cA_t$ for $\cA_X$.  Note that the ancestor
potentials \eqref{eq:ancestor} do not contain terms with $g=0$ and
$m<3$, or with $g=1$ and $m=0$, as in these cases the space
$\cMbar_{g,m}$ is empty and so the map $p_m\colon X_{g,m+n,d} \to
\cMbar_{g,m}$ is not defined.

\subsection{Dilaton Shift} 
\label{sec:dilaton} 

Consider now another sequence $\bq =(q_0,q_1,q_2,\dots)$ with $q_i \in
H_X$, and write $q_i = q_i^0 \phi_0 + \cdots + q_i^N \phi_N$.  We
regard $\{q_i^\alpha : 0 \leq i < \infty, 0 \leq \alpha \leq N \}$ as a
co-ordinate system on $H_X[\![z]\!]$, by writing a general point in
$H_X[\![z]\!]$ as $\bq(z) = \sum_{i=0}^\infty q_i z^i$.  The
\emph{dilaton shift} is an identification between $\bq
=(q_0,q_1,q_2,\dots)$ and the arguments $\bt =(t_0,t_1,t_2,\dots)$,
$\by =(y_0,y_1,y_2,\dots)$ of the descendant and ancestor potentials:
\begin{align} 
 & q_i^\alpha = 
  \begin{cases}
    t_1^0 - 1 & \text{if $(i,\alpha) = (1,0)$} \\
    t_i^\alpha & \text{otherwise}
  \end{cases}
  && q_i^\alpha = 
  \begin{cases}
    y_1^0 - 1 & \text{if $(i,\alpha) = (1,0)$} \\
    y_i^\alpha & \text{otherwise}
  \end{cases}
  \notag \\
  \intertext{Setting $\bt(z) = \sum_{i=0}^\infty t_i z^i$ and $\by(z) =
    \sum_{i=0}^\infty y_i z^i$, the dilaton shift becomes the equalities:}
  \label{eq:dilatonshift} 
  & \bq(z) = \bt(z) - \phi_0 z
  && \bq(z) = \by(z) - \phi_0 z
\end{align}
In this way we regard the descendant potential $\cF^g_X$ as a function
on the formal neighbourhood of the point ${-\phi_0} z \in H_X$.  The
dilaton shift for the ancestor potential is discussed in
Example~\ref{exa:AX}.

\subsection{The Orbifold Case} 
The results in this paper are all valid in the more general setting
where $X$ is a smooth orbifold (or Deligne--Mumford stack) rather than
a smooth algebraic variety.  The discussion above goes through in this
situation with minimal changes, as follows:
\begin{itemize}
\item We take $H_X$ to be the even part\footnote{Here we mean the even
    part of the rational cohomology of the inertia stack $IX$ with
    respect to the usual grading on $H^\bullet(IX)$, not the age-shifted
    grading.} of the Chen--Ruan orbifold cohomology
  $H^\bullet_{\CR}(X;\Q)$ rather than the even part of the ordinary
  cohomology $H^\bullet(X;\Q)$.
\item We replace:
  \begin{itemize}
  \item the usual grading on $H^\bullet(X)$
    by the age-shifted grading on $H^\bullet_{\CR}(X)$
 \item the Poincar\'e pairing on $H^\bullet(X)$ by the orbifold
    Poincar\'e pairing on $H^\bullet_{\CR}(X)$.
  \end{itemize}
  Note that $H^2(X) \subset H^2_{\CR}(X)$, and so definition
  \eqref{eq:basisproperties} makes sense in the orbifold context.
\item We define correlators \eqref{eq:correlator} and
  \eqref{eq:mixedcorrelator} using orbifold Gromov--Witten invariants
  \cite{AGV} rather than usual Gromov--Witten invariants.  
  There are two small differences:
  \begin{itemize}
  \item a subtlety in the definition of $\ev_k^\star$, 
  discussed in \cite{AGV}, \cite{CCLT}*{\S2.2.2}
  \item the degree $d$ of an orbifold stable map $f:\Sigma \to X$ 
    lies in $H_2(|X|;\Z)$, where $|X|$ is the coarse moduli 
    space of $X$. 
  \end{itemize}
\end{itemize}
Having made these changes, the discussion in
\S\S\ref{sec:GW}--\ref{sec:dilaton} 
applies to orbifolds as well.  
In this context, the family of algebras 
$\big(H_X \otimes\Lambda,\ast\big)$ 
is called \emph{quantum orbifold cohomology}. 

\subsection{FJRW Theory}
The discussion in this paper applies also to the so-called
FJRW~theory, which has been developed recently by Fan--Jarvis--Ruan
based on an old idea of Witten
\citelist{\cite{FJR}\cite{Witten:algebraic}}.  FJRW theory is a
Gromov--Witten-type theory with target a Landau--Ginzburg orbifold: it
defines a Cohomological Field Theory (CohFT) on a certain state space
$H_{\FJRW}$ which satisfies Teleman's homogeneity and flat vacuum
conditions. Thus Teleman's classification result applies to FJRW
theory.  FJRW theory differs from Gromov--Witten theory in that it
lacks Novikov variables $Q_1,\ldots,Q_r$; most of the discussion in this paper,
however, goes through just by setting $r=0$:
\begin{itemize} 
\item The genus-zero part of FJRW theory defines a Frobenius manifold
  structure on the formal neighbourhood of the origin of $H_{\FJRW}$;
\item Formal Semisimplicity \eqref{eq:formalsemisimplicity},
  Genus-Zero Convegence \eqref{eq:genuszeroconvergence} and Analytic
  Semisimplicity \eqref{eq:analyticsemisimplicity} make sense for this
  Frobenius manifold;
\item The descendant potential $\cZ_{\FJRW}$ is a formal power series
  in $\bt(z) \in H_{\FJRW}[\![z]\!]$;
\item the ancestor potential $\cA_{\FJRW,t}$ is a formal power series
  in $\by(z)\in H_{\FJRW}[\![z]\!]$ and $t\in H_{\FJRW}$.
\end{itemize}

\section{Givental's Quantization Formalism} 
\label{sec:quantization}

In this section, we work over an arbitrary commutative ring $R$ which
contains $\Q$.  Let $V$ be a finitely generated free $R$-module
equipped with a symmetric perfect pairing:
\[
\pair{\cdot}{\cdot}_V \colon 
V\otimes_R V \to R.  
\]
Let $\{\phi_\alpha\}_{\alpha=0}^N$ be an $R$-basis of $V$ and let
$\phi^\alpha$ be the dual basis with respect to the pairing
$\pair{\cdot}{\cdot}_V$, so that $\pair{\phi_\alpha}{\phi^\beta}_V =
\delta_\alpha^\beta$.  We denote a general point of $V[\![z]\!]$ by:
\[
\bq(z) = q_0 + q_1 z + q_2 z + q_3 z^3 + \cdots
\]
and write $q_i = q_i^0 \phi_0 + \cdots + q_i^N \phi_N$. Then $\{
q_i^\alpha :  0 \leq i < \infty, 0\le \alpha \le N \}$ gives a co-ordinate system
on $V[\![z]\!]$.

\begin{remark}
  \label{rem:HXexample}
  In the case where $R = \Q$, $V = H_X$, and $\pair{\cdot}{\cdot}_V$
  is the Poincar\'e pairing, we recover the situation described in
  \S\ref{sec:dilaton}.  
\end{remark}

\subsection{Ancestor Fock Space}

\begin{definition}[Ancestor Fock Space; see Givental \cite{Givental:quantization}*{\S8}]
  \label{def:ancestorFock}
  Choose a base point $-\delta = - \sum_{\alpha=0}^N \delta^\alpha
  \phi_\alpha \in V$, and consider the co-ordinate system $\{
  y_i^\alpha : 0 \leq i < \infty, 0\le \alpha \le N \}$ on
  $V[\![z]\!]$ defined by:
  \[
  y_i^\alpha = 
  \begin{cases}
    q_1^\alpha + \delta^\alpha & \text{if $i=1$} \\
    q_i^\alpha & \text{otherwise}
  \end{cases}
  \]
  Let $R[\![\by]\!]$ denote the formal power series ring
  $R[\![y_i^\alpha : 0 \leq i < \infty, 0 \leq \alpha \leq N]\!]$
  equipped with the valuation $v$ defined by $v(y_i^\alpha) = i+1$.
  The \emph{ancestor Fock space} $\Fock(V,\delta)$ is the set of
  elements
  \[
  \cA \in R[\![\by]\!]\{\hbar^{-1},\hbar]\!]
  \]
  that admit an expansion of the form: 
  \begin{equation}
    \label{eq:Fockexpansion}
    \cA = \exp \left( \sum_{g=0}^\infty \hbar^{g-1} \cF^{g}
    \right)
  \end{equation}
  such that $\cF^g \in R[\![\by]\!]$ and: 
\begin{align}
\label{eq:tameness}
\begin{split} 
& \left. \cF^0 \right|_{\by(z)=0} = 
\left.
\parfrac{\cF^0}{y_i^\alpha} \right |_{\by(z)=0}= 
\left. \parfrac{^2 \cF^0}{y_{i_1}^{\alpha_1} \partial y_{i_2}^{\alpha_2}} 
\right|_{\by(z)=0}=0,  \quad 
\left. \cF^1 \right|_{\by(z)=0} = 0 
\\
& {\partial^n \cF^g \over \partial
      y_{i_1}^{\alpha_1} \cdots \partial y_{i_n}^{\alpha_n}}
    \bigg|_{\by(z) = 0} = 0 \quad \text{ whenever}
\quad i_1 + \cdots +
    i_n > 3g-3+n.
\end{split} 
\end{align}
\end{definition}  
Write $y_i = y_i^0 \phi_0 + \cdots + y_i^N \phi_N$ and $\by(z) =
\sum_{i=0}^\infty y_i z^i$.  The co-ordinate system $\by =
(y_0,y_1,y_2,\ldots)$ from Definition~\ref{def:ancestorFock} is
related to the co-ordinate system $\bq = (q_0,q_1,q_2,\ldots)$ defined
above Remark~\ref{rem:HXexample} by:
\begin{equation}
  \label{eq:generaldilatonshift}
 q_i^\alpha = 
  \begin{cases}
    y_1^\alpha - \delta^\alpha & \text{if $i=1$} \\
    y_i^\alpha & \text{otherwise}
  \end{cases}
\end{equation}
or in other words by $\bq(z) = \by(z) - \delta z$; cf.~the dilaton
shift \eqref{eq:dilatonshift}.  Elements of $\Fock(V,\delta)$ can thus
be regarded as functions on a formal neighbourhood of the point ${-
  \delta} z \in V[\![z]\!]$.  

\begin{remark}
  Any expression of the form \eqref{eq:Fockexpansion} such that $\cF^g
  \in R[\![\by]\!]$ and condition
  \eqref{eq:tameness} holds is automatically an element of
  $R[\![\by]\!]\{\hbar^{-1},\hbar]\!]$.
\end{remark}

\begin{remark}
  Condition \eqref{eq:tameness} implies that any element $\cA$ of
  $\Fock(V,\delta)$ is \emph{tame} in the sense of Givental
  \cite{Givental:An}.  Note in particular that $\cF^g$ is a formal
  power series in the variables $y_0^0,\dots,y_0^N,y_1^0,\dots,y_1^N$
  with coefficients in the polynomial ring $R[y_i^\alpha : 2\le i<
  \infty, \, 0\le \alpha \le N]$.
\end{remark}

\begin{definition}(Rationality) 
\label{def:rationality} 
An element $\cA$ of $\Fock(V,\delta)$
  is called \emph{rational} if there exists a polynomial $P(q_1) \in
  R[V^\vee]$ with $P(-\delta) = 1$ 
such that the potentials $\cF^g$ from \eqref{eq:Fockexpansion} satisfy:
 \begin{equation} 
 \label{eq:rationality}
  \frac{\partial^n \cF^g}
  { \partial y_{i_1}^{\alpha_1} \cdots \partial y_{i_n}^{\alpha_n}}
  \bigg|_{\by(z) = y_1 z} 
  = 
f_{g,I,A}(q_1) P(q_1)^{-(5g-5+2 n- (i_1+\cdots +i_n))}  
  \end{equation} 
for some polynomials $f_{g,I,A}(q_1) \in R[V^\vee]$ 
if $2g-2+n>0$; 
here $I = (i_1,\ldots,i_n)$, $A = (\alpha_1,\ldots,\alpha_n)$. 
We call $P$ the \emph{discriminant} of $\cA$.  
\end{definition} 

\begin{remark} 
\label{rem:rationality} 
Tameness \eqref{eq:tameness} and rationality 
\eqref{eq:rationality} for a potential can be  
summarized in the following expansion: 
\begin{align*}
\cF^{g}= 
\delta_{g,1} C^{(1)}(q_1) + 
\sum_{n:2g-2+n > 0}  
  \frac{1}{n!}
  \sum_{\substack{ 
      I : I = (i_1,\ldots,i_n) \\
      \text{$i_j \neq 1$ for all $j$} \\ 
      i_1 + \cdots + i_n \leq 3 g -3 + n}} 
 \sum_{A = (\alpha_1,\ldots,\alpha_n)}
 C^{(g)}_{I,A}(q_1) \,
  q_{i_1}^{\alpha_1}\cdots q_{i_n}^{\alpha_n}
\end{align*}
with 
\begin{align} 
\label{eq:coeff-rationality}
\begin{split} 
C^{(g)}_{I,A} (q_1) 
& = 
f_{g,I,A}(q_1) P(q_1)^{-(5g-5+2n
-(i_1+\cdots+i_n))}  \\ 
\parfrac{C^{(1)}(q_1)}{q_1^\alpha} 
& = 
f_{1,1,\alpha}(q_1) P(q_1)^{-1}
\end{split}  
\end{align} 
for some polynomials $f_{g,I,A}(q_1) \in R[V^\vee]$. 
Note that $5g-5+2n-(i_1+\cdots+i_n) 
= 3g-3+n - (i_1+ \cdots +i_n) +2g-2+n$ is always 
positive unless $(g,n)=(1,0)$.  
The genus-one term $C^{(1)}(q_1)$ is in general not 
a rational function. 
See Remark \ref{rem:genus-one-log} below 
in the case of Gromov--Witten theory. 
\end{remark} 

\begin{example} 
  \label{exa:AX}
  The total ancestor potential $\cA_X$ of $X$ defines an element of
  the Fock space $\Fock(H_X\otimes \Lambda[\![t]\!], \phi_0)$.  Here
  the ground ring $R$ is $\Lambda[\![t]\!]$; the $R$-module $V$ is
  $H_X\otimes \Lambda[\![t]\!]$; and the pairing
  $\pair{\cdot}{\cdot}_V$ is the Poincar\'e pairing, extended by
  $R$-linearity to take values in $R$.  The dilaton shift discussed in
  \S\ref{sec:dilaton} coincides with the identification
  \eqref{eq:generaldilatonshift}.  Tameness \eqref{eq:tameness}
  follows from the dimension formula $\dim \cMbar_{g,m} = 3g-3 +m$.
\end{example}

\begin{remark} 
\label{rem:genus-one-log} 
The genus-one ancestor potential of a smooth projective 
variety $X$ satisfies \cite{Dijkgraaf-Witten:meanfield}: 
\[
\bar\cF^{1}_t \Bigr|_{\by(z)=y_1 z} = 
-\frac{1}{24} \log  \sdet(-q_1*_t)  
\]
where $\sdet(-q_1*_t)$ denotes the superdeterminant 
of the quantum product on the 
total cohomology group $H^\bullet(X) = 
H^{\rm even}(X) \oplus H^{\rm odd}(X)$ 
(including the odd part). 
This follows from the localization of the integral 
to the locus of cycles of rational curves and 
$\int_{\cMbar_{1,1}} \psi = 1/24$. 
Therefore the genus-one potential itself is not rational 
in $q_1$, but its derivatives are rational. 
\end{remark} 

\begin{example}
  \label{exa:Apt}
  The ancestor potential $\cA_{\rm pt} = \cA_t$ 
  of a point does not depend on $t\in
  H_{\rm pt}$ and coincides with the descendant potential $\cZ_{\rm
    pt}$.  This is called the \emph{Witten--Kontsevich tau-function} and
  denoted by $\tau(\bq)$.  It defines a rational element of the Fock
  space with $V=R=\C$ and $\delta =1$.  
In fact, applying the Dilaton Equation, we find that:
  \[
  \frac{\partial^n \cF^g_{\text{\rm pt}}}
  { \partial y_{i_1} \cdots \partial y_{i_n}}
  \bigg|_{\by(z) = y_1 z} 
  =
  \begin{dcases}
    {- \frac{1}{24}} \log (-q_1) & \text{if $g=1$ and $n=0$} \\
    (-q_1)^{-(2g-2+n)} \corr{\psi_1^{i_1},\dots,\psi_n^{i_n}}_{g,n,0}^{\rm
      pt} & \text{otherwise}
  \end{dcases}
  \]
Hence we can take $P(q_1) = -q_1$. 
Note that $i_1+ \cdots +i_n =3g-3+n$ implies 
$2g-2+n \le 5g-5+2n- (i_1+\cdots+i_n)$. 
\end{example} 

\begin{remark}
  In view of Givental's formula (see
  \S\S\ref{sec:Giventalformulaanalytic}--\ref{sec:Giventalformulaformal})
  one may speculate that in general the total ancestor potential of
  $X$ is rational with discriminant
  $\det(-q_1*)$ (determinant on the even part $H_X$ 
with even parameter $t\in H_X$ and even $q_1\in H_X$). 
We will prove that this is the case whenever the
quantum cohomology of $X$ is semisimple: see Theorem
\ref{thm:Ageom=Aabs}.
\end{remark}

\begin{remark}
  Givental's Lagrangian cone $\mathcal{L}_X$ (see \cite{Givental:symplectic})
  has a singularity along a ``divisor'' which contains the vertex of the
  cone. Thus it is natural to conjecture that the higher genus
  descendant potentials of $X$ has poles only along that
  divisor. This is the rationality condition. 
\end{remark}

\begin{remark}
  \label{rem:ancestordivisor}
  Recall the definition of the genus-$g$ ancestor potential
  $\bar{\cF}^g_{X}$ in \eqref{eq:ancestor}.  Consider the completion
  $\cS$ of the polynomial ring $\Q\big[t^0,Q_1 e^{t^1}, \ldots, Q_r
  e^{t^r}, t^{r+1},t^{r+2},\ldots,t^N\big]$ with respect to the
  valuation $v$ defined by:
  \begin{align*}
    & v(t_i) = 1 && \text{$i=0$ or $r<i\leq N$}\\
    & v(Q_i e^{t^i}) = 1 && 1 \leq i \leq r 
  \end{align*}
  The Divisor Equation implies that $\bar{\cF}^g_{X}$, which \emph{a
    priori} is a formal power series in the variables $y_j^\beta$ with
  coefficients in:
  \[
  \Q\big[\!\big[Q_1,\ldots,Q_r\big]\!\big]\big[\!\big[t^0,\ldots,t^N\big]\!\big]
  \]
  is in fact a formal power series in the variables $y_j^\beta$ with
  coefficients in $\cS$.  Thus the total ancestor potential $\cA_X$
  defines an element of the Fock space $\Fock(H_X\otimes \cS,\phi_0)$.
\end{remark}

\begin{definition} 
\label{def:ancestor-converge} 
For $\epsilon >0$, define $\cS_\epsilon$ to be the subring of $\cS$
consisting of elements in $\cS$ which converge on the region:
\begin{equation} 
\label{eq:e-ball} 
\left\{|t^0|<\epsilon,\ |Q_1 e^{t^1}|<\epsilon, 
\cdots, |Q_r e^{t^r}|<\epsilon, \ 
|t^{r+1}|<\epsilon,\cdots, 
|t^{N}|<\epsilon \right\}.  
\end{equation} 
The ancestor Gromov--Witten potential $\cA_X$ is said to be
\emph{convergent} if it is a rational element of $\Fock(H_X\otimes
\cS_\epsilon, \phi_0)$ for some $\epsilon>0$.
\end{definition} 

\begin{remark}
  When $\cA_X$ is convergent in the sense of
  Definition~\ref{def:ancestor-converge}, each genus-$g$ ancestor
  potential $\bar{\cF}^g_X$ (see equation \ref{eq:ancestor}) is a
  power series in the variables $y_j^\beta$ with coefficients in
  $\cS_\epsilon$.  Furthermore in this case
  $\bar{\cF}^g_X|_{Q_1=\cdots=Q_r=1}$ is a formal power series in
  $y_j^\beta$ with coefficients in analytic functions on $\cM$, where
  $\cM$ is a neighbourhood \eqref{eq:LRLnbhd} of the large-radius
  limit point.
\end{remark}

\subsection{Propagator} 
Let $(V,\pair{\cdot}{\cdot}_V)$, 
$(W,\pair{\cdot}{\cdot}_W)$ be free $R$-modules 
with symmetric perfect pairings. 

\begin{definition}
  \emph{The Givental symplectic form} $\Omega_V$ is an antisymmetric
  bilinear form on $V(\!(z)\!)$ defined by:
  \begin{equation} 
    \label{eq:symplecticform} 
    \Omega_V(f_1, f_2) = \Res_{z=0} \pair{f_1(-z)}{f_2(z)}_V dz. 
  \end{equation} 
\end{definition}

\begin{notation}
  \label{not:A0}
  An $R[\![z]\!]$-linear isomorphism $A:V[\![z]\!] \to W[\![z]\!]$ can
  be expressed uniquely in the form $A = A_0 + A_1 z + A_2 z^2 +
  \cdots$ where $A_k \in \Hom_R(V,W)$.  We write the coefficients of
  this expansion as $A_k$, $k \geq 0$, and write $A$ as $A(z)$ when we
  wish to emphasize the dependence on $z$.
\end{notation}

\begin{definition} 
\label{def:unitarity} 
An isomorphism $A \colon V[\![z]\!]  \to W[\![z]\!]$ is said to be
\emph{unitary} if it is $R[\![z]\!]$-linear and satisfies:
\[
\pair{A(-z) v_1}{A(z) v_2}_W = \pair{v_1}{v_2}_V. 
\]
for all $v_1, v_2 \in V$.
\end{definition}

\begin{remark}
  An $R[\![z]\!]$-linear isomorphism $A \colon V[\![z]\!]  \to
  W[\![z]\!]$ is unitary if and only if the map $V(\!(z)\!) \to
  W(\!(z)\!)$ induced by $A$ intertwines the Givental symplectic
  forms.
\end{remark}

\begin{definition}[Propagator; cf.~Givental \cite{Givental:quantization}]
  Let $A:V[\![z]\!] \to W[\![z]\!]$ be a unitary isomorphism.  The
  \emph{propagator} for $A$ is a bivector field $\Delta$ on
  $V[\![z]\!]$ defined by
  \[
  \Delta = \sum_{i, j=0}^\infty 
  \sum_{\alpha,\beta = 0}^N
  \Delta^{(i,\alpha), (j,\beta)} \parfrac{}{q_i^\alpha} 
  \parfrac{}{q_j^\beta}
  \]
  where:
  \[
  \sum_{i, j=0}^\infty
  \Delta^{(i,\alpha),(j,\beta)} (-1)^{i+j}
  w^i z^j = \Pair{\phi^\alpha}{ 
    \frac{A(w)^\dagger A(z) - \Id}{z+w} \phi^\beta}_V
  \]
  Here the co-ordinates $q_i^\alpha$ and the basis $\{\phi^\alpha\}$
  are defined above Remark~\ref{rem:HXexample}; $\Delta$ is in fact
  independent of choice of basis.
\end{definition}

\subsection{Quantized Operator}
Let $A \colon V[\![z]\!] \to W[\![z]\!]$ be a unitary isomorphism.
Recall the definition of $A_0$ in Notation~\ref{not:A0} above.  We
define the quantized operator
\[
\widehat{A}: \Fock(V, \delta) \to \Fock(W, A_0(\delta))
\]
as follows.  For a given element $\cA \in \Fock(V,\delta)$, we set:
\[
\tcA = \exp\left(\textstyle\frac{\hbar}{2} \Delta \right) \cA 
\in \Fock(V,\delta) 
\]
and then push $\tcA$ forward along the identification $A(z) \colon
V[\![z]\!] \cong W[\![z]\!]$
\[
(\widehat{A} \cA) ( \bq ) := \tcA(A(z)^{-1} \bq(z)).
\]

\begin{proposition}
  \label{pro:Feynman}
  The quantized operator $\widehat{A}$ is well-defined.  Moreover, if
  $\cA$ is a rational element of $\Fock(V, \delta)$ 
  with discriminant $P(q_1) \in R[V^\vee]$ then $\widehat{A}\cA$ is a
  rational element of $\Fock(W, A_0(\delta))$ with 
  discriminant $P(A_0^{-1}q_1) \in R[W^\vee]$
\end{proposition} 

\begin{proof}
  The first claim was proved by Givental using a Feynman diagram
  argument \cite{Givental:An}*{Proposition~5}.  It remains to show
  that the quantized operator $\widehat{A}$ preserves rationality, and
  to calculate its effect on the discriminant.  Recall that
  $\tcA = \exp\left(\textstyle\frac{\hbar}{2} \Delta \right) \cA$, and
  define $\tcF^g$ by:
  \[
  \tcA = \exp \left( \sum_{g=0}^\infty \hbar^{g-1} \tcF^{g} \right)
  \]
  Following Givental's proof, we express:
  \begin{equation}
    \label{eq:partialtcFg}
    \frac{\partial^n \tcF^g}
    { \partial y_{i_1}^{\alpha_1} \cdots \partial y_{i_n}^{\alpha_n}}
    \bigg|_{\by(z) = y_1 z} 
  \end{equation}
  as a sum over decorated Feynman graphs.  These decorated Feynman
  graphs are connected multigraphs, in which loops are allowed, such
  that:
  \begin{itemize} 
  \item each vertex $v$ is labelled by an integer $g_v\ge 0$;
  \item a label $(j,\beta)\in \Z_{\ge 0}\times \{0,\dots,N\}$ is
    assigned to each pair of a vertex and an edge incident to it (for
    an edge-loop, we distinguish the two ends of the edge);
  \item the graph has $n$ external edges, called \emph{legs}, labelled
    by $(i_1,\alpha_1),\dots,(i_n,\alpha_n)$;
  \item the Euler number $\chi$ of the graph satisfies $g = 1- \chi +
    \sum_{v:\text{vertex}} g_v $.
  \end{itemize} 
  and such that the following stability condition holds: for each
  vertex $v$, if $(j_1,\beta_1),\dots,(j_m,\beta_m)$ are all the
  labels attached to the edges or legs incident to $v$, then:
  \[
  j_1 + \dots + j_m \le 3 g_v -3 + m
  \]
  Givental's original argument shows that the number of such decorated
  Feynman graphs is finite.  Let $\Gamma$ be a decorated Feynman graph
  as above, and let $V(\Gamma)$, $E(\Gamma)$ be respectively the set
  of vertices and the set of edges of $\Gamma$.  The contribution of
  $\Gamma$ to \eqref{eq:partialtcFg} is:
  \begin{equation}
    \label{eq:Gammacontribution}
    \frac{1}{|\Aut(\Gamma)|} \prod_{e \in E(\Gamma)} (\text{edge term
      for $e$}) \prod_{v \in V(\Gamma)} (\text{vertex term for $v$})
  \end{equation}
  where the edge term for an edge with labels $(i,\alpha)$, $(j,\beta)$
  is $\Delta^{(i,\alpha),(j,\beta)}$, and the vertex term for a vertex
  $v$ with labels $(j_1,\beta_1)$,\dots,$(j_m,\beta_m)$ is:
  \begin{equation} 
    \label{eq:derivativeofFg} 
    \frac{\partial^n \cF^{g_v}}{\partial y_{j_{1}}^{\beta_{1}}
      \cdots \partial y_{j_m}^{\beta_m}}\Bigg|_{\by(z) =y_1 z}
  \end{equation} 
  We write $n_v= m$ and $d_v = j_1+ \cdots + j_m$. 
  Suppose that $\cA$ is rational with discriminant
  $P(q_1)$.  
The partial derivative \eqref{eq:partialtcFg}
  is a finite sum of terms \eqref{eq:Gammacontribution}, and each
  vertex term \eqref{eq:derivativeofFg} takes the form:
  \begin{equation}
    \label{eq:theformwewant}
    \frac{f_v(q_1)}{P(q_1)^{5g_v-5+2n_v- d_v}}
  \end{equation}
  where $f_v$ is a polynomial. 
Using the (in)equalities: 
\[
\sum_{v\in V(\Gamma)} (g_v- 1)  = g-1 - |E(\Gamma)|,  
\quad \sum_{v\in V(\Gamma)} n_v  =2 |E(\Gamma)| + n, 
\quad 
\sum_{v\in V(\Gamma)} d_v \ge i_1+\cdots +i_n 
\]
we have 
\[
\sum_{v\in V(\Gamma)} (5g_v -5 + 2n_v - d_v)
\le 5g-5 + 2n - (i_1+ \cdots + i_n)
\]
Hence each term \eqref{eq:Gammacontribution} is a rational function with 
denominator $P(q_1)^{5g-5+2n - (i_1+\cdots +i_n)}$. 
It follows that $\tcA$ is rational with
discriminant $P(q_1) \in R[V^\vee]$.  
The change of variables $\bq(z) \to A(z)^{-1} \bq(z)$ 
preserves tameness and rationality: one can easily 
check that the expansion in Remark \ref{rem:rationality} 
is preserved. 
Thus $\widehat{A}\cA$ is rational, with 
discriminant $P(A_0^{-1}q_1) \in R[W^\vee]$.
\end{proof} 

\begin{example} Figure~1 below shows an example of a decorated Feyman
  diagram $\Gamma$.

  \begin{figure}[htbp]
    \begin{center} 
      \begin{picture}(300,50)(-30,0) 
        \put(0,25){\circle*{3}} 
        \put(100,25){\circle*{3}}
        \put(200,25){\circle*{3}}
        {\thicklines
          \put(225,25){\circle{50}} 
          \path(0,25)(200,25) 
          \put(100,25){\vector(0,-1){20}}
        }
        \put(0,32){\makebox(0,0){$g_1$}}
        \put(100,32){\makebox(0,0){$g_2$}} 
        \put(193,32){\makebox(0,0){$g_3$}} 
        \put(15,18){\makebox(0,0){\small $(i,\alpha)$}} 
        \put(82,18){\makebox(0,0){\small $(j,\beta)$}}
        \put(118,18){\makebox(0,0){\small $(k,\gamma)$}}
        \put(183,18){\makebox(0,0){\small $(l,\epsilon)$}}
       \put(218,36){\makebox(0,0){\small $(m,\rho)$}}
        \put(218,14){\makebox(0,0){\small $(n,\mu)$}}
        \put(100,0){\makebox(0,0){\small $(p,\xi)$}}
      \end{picture} 
    \end{center} 
    \caption{A decorated graph with one leg.} 
    \label{fig:decoratedgraph} 
  \end{figure}
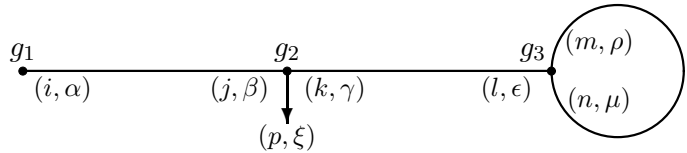

  \noindent This graph $\Gamma$ has one leg, labelled by $(p,\xi)$; it
  occurs in the Feynman sum for:
  \begin{align*}
    \frac{\partial \tcF^{g}}{\partial y_{p}^{\xi}} \Bigg|_{\by(z) =y_1
      z} && \text{where $g = g_1+g_2+g_3+1$.}  
  \end{align*}
  The stability condition asserts that $i \le 3g_1-2$, $j + k + p \le
  3 g_2$, and $l+m+n\le 3g_3$.  The automorphism group of $\Gamma$ is
  trivial if $(m,\rho) \ne (n,\mu)$, and is equal to $\Z/2\Z$ if
  $(m,\rho) = (n,\mu)$.  Thus the contribution of $\Gamma$ to
  the Feynman sum is equal to:
  \[
  \Delta^{(i,\alpha),(j,\beta)} 
  \Delta^{(k,\gamma),(l,\epsilon)} 
  \Delta^{(m,\rho),(n,\mu)}
  \Bigg(
  \frac{\partial \cF^{g_1}}{\partial y_i^\alpha} 
  \frac{\partial^3 \cF^{g_2}}
  {\partial y_j^\beta \partial y_k^\gamma \partial y_p^\xi}  
  \frac{\partial^3 \cF^{g_3}}
  {\partial y_l^\epsilon \partial y_m^\rho \partial y_n^\mu} 
  \Bigg)\Bigg|_{\by(z)=y_1 z}.
  \]
  if $(m,\rho) \ne (n,\mu)$, and is equal to half of this if $(m,\rho)
  = (n,\mu)$.
\end{example}

\begin{remark} 
  Let $(U,\pair{\cdot}{\cdot}_U)$ be another free $R$-module with a
  perfect pairing.  Let $A \colon V[\![z]\!] \to W[\![z]\!]$ and $B
  \colon W[\![z]\!] \to U[\![z]\!]$ be unitary isomorphisms.  Then one
  can define three propagators $\Delta^A$, $\Delta^{B}$, $\Delta^{BA}$
  corresponding to the maps $A$, $B$, $BA$ respectively.
  The bivector fields $\Delta^A$ on $V[\![z]\!]$, $\Delta^{B}$ on $W[\![z]\!]$, and
  $\Delta^{BA}$ on $V[\![z]\!]$ satisfy:
  \[
  \Delta^{BA} = \Delta^A + A(z)^* \Delta^B. 
  \]
  Therefore 
  \[
  (BA)\sphat = \widehat{B} \widehat{A}
  \]
  as a map from $\Fock(V,\delta)$ to $\Fock(U,B_0A_0(\delta))$.  
\end{remark}

\section{Givental's formula in the Analytic Setting}
\label{sec:Giventalformulaanalytic}
Let $\cM$ be an analytic Frobenius manifold over $\C$. 
This comprises the following
data: a smooth complex analytic space $\cM$; 
a flat metric\footnote
{Metric here means $\C$-bilinear quadratic form on each
tangent space $T_t M$, varying analytically with $t$.} $g$ on $\cM$;
a product $\ast_t$ on each tangent space $T_t \cM$, varying
analytically with $t$; 
a flat identity vector field $\unit$; 
a vector field $E$ on $\cM$ called the Euler
vector field; and an integer $D$ called the conformal dimension.
These structures are required to satisfy a number of conditions: see
\cite{Dubrovin:2DTFT}*{Definition~1.2}.  In particular $(T_t \cM,
\ast_t, g)$ forms a family of commutative associative Frobenius
algebras, varying analytically with $t$, and $\nabla^{\LC}
\big(\nabla^{\LC} E\big) = 0$ where $\nabla^{\LC}$ is the
Levi--Civita connection defined by $g$.  The operator $\mu:T\cM \to
T\cM$ defined by $\mu = \big(1-\frac{D}{2}\big)\Id - \nabla^{\LC} E$
is called the grading operator.  One example of an analytic Frobenius
manifold over $\C$ is given by the quantum cohomology of a smooth
variety $X$ such that the genus-zero Gromov--Witten potential
converges in the sense of \S\ref{sec:convergence}; in this case $\cM$
is the neighbourhood \eqref{eq:LRLnbhd} of the large-radius limit
point.

Suppose further that $\cM$ is \emph{generically semisimple}, 
i.e.~that $\big(T_t \cM,\ast_t\big)$ is a semisimple algebra for generic
$t \in \cM$, and fix a semisimple point $t$. 
The eigenvalues of multiplication $\big({E\ast}\big)$ 
by the Euler vector field form \emph{canonical co-ordinates} 
$u^0, \ldots, u^N$ on a neighbourhood of $t$. 
The vector fields $\parfrac{}{u^i} \in T\cM$ are then the idempotents
in the semisimple algebra $\big(T\cM,\ast \big)$ 
in a neighbourhood of $t$. 
Let:
\[
\Delta^i(t) = {1 \over
  g\Big(\left.\parfrac{}{u^i}\right|_t,\left.\parfrac{}{u^i}\right|_t\Big)}
\]

\begin{proposition}[Dubrovin \protect{\cite[Lecture~4]{Dubrovin:Painleve}},
  Teleman \protect{\cite[Theorem~8.15]{Teleman}}]
  \label{pro:R}
  At the semisimple point $t \in \cM$, the
  equation:
  \[
  \Bigg( z \parfrac{}{z} + \frac{1}{z} \big({E
  \ast_t}\big) + \mu \Bigg) S = 0
  \]
  has a unique solution of the form $S = \Psi_t R_t \exp(U/z)$ such that:
  \begin{enumerate}
  \item $\Psi_t \in \Hom\big(\C^{N+1},T_t \cM\big)$ is the isomorphism
    $\C^{N+1} \cong T_t \cM$ that sends the $i$th standard basis vector
    in $\C^{N+1}$ to the $i$th normalized idempotent 
    $\sqrt{\Delta^i(t)} \parfrac{}{u^i} \in T_t \cM$
  \item $R_t \in \End\big(\C^{N+1}\big) \otimes \C[\![z]\!]$ with $R_t
    \equiv \Id \mod z$
  \item $U = \operatorname{diag}\big(u^0,\ldots,u^N\big)$ where 
  $u^0,\dots,u^N$ are the eigenvalues of $E*_t$. 
  \end{enumerate}
  The transformation $R_t$ satisfies:
  \[
  R_t(-z)^{\mathrm{T}} R_t(z) = \Id
  \]
\end{proposition}

The transformations $\Psi$ and $R$ in Proposition~\ref{pro:R} coincide
with those defined by Givental \cite{Givental:semisimple}*{\S1.3}, 
although his definitions are 
different as he is working in a setting where there may be no Euler
vector field.  As Dubrovin observed, $\Psi_t R_t \exp(U/z)$ is
automatically flat with respect to the Dubrovin connection 
as $t$ varies and, as $t$ varies, 
$R_t$ is automatically homogeneous with respect to 
the Euler vector field $E=\sum_{i=0}^N u^i \parfrac{}{u^i}$: 
\[
\Bigg(z \parfrac{}{z} + \sum_i u^i \parfrac{}{u^i}\Bigg) R_t = 0
\]

We regard the composite map $\Psi_t R_t$ as giving a 
unitary isomorphism
$\C^{N+1} [\![z]\!] \to T_t\cM[\![z]\!]$ 
where $\C^{N+1}$ is endowed with the standard inner product 
(see Definition \ref{def:unitarity}). 
In view of Example~\ref{exa:Apt}, 
we know that the product of Witten--Kontsevich 
$\tau$-functions:
\[
\cT = \prod_{\alpha=0}^{N} \tau(\bq^\alpha) \qquad 
\text{where} \quad  
(\bq^0,\dots,\bq^N) \in \C^{N+1} [\![z]\!] 
\] 
lies in the Fock space $\Fock(\C^{N+1},(1,\dots,1))$. 
It is rational with with the discriminant 
$P(q_1^0,\dots,q_1^N) = \prod_{\alpha=0}^N (-q_1^\alpha)$. 

\begin{definition}[Givental \protect{\cite[\S6.8]{Givental:quantization}}]
  \label{def:ancestor}
  The \emph{abstract ancestor potential} $\cAabs_t$ is:
  \begin{equation}
    \label{eq:Aabs}
    \cAabs_t = e^{{-\frac{1}{48}} \sum_i \log
      \Delta^i(t)}\,
    \widehat{\Psi_t} \widehat{R_t} (\cT)
  \end{equation}
  When the semisimple point $t \in \cM$ is clear from context, we will
  write $\cAabs$ instead of $\cAabs_t$.  
\end{definition}

\begin{proposition}
  \label{pro:Aabs}
  The abstract ancestor potential $\cAabs_t$ is a well-defined
 rational element of $\Fock(T_t \cM, \unit)$, with 
discriminant $\det(-q_1*_t)$.
\end{proposition}

\begin{proof}
  We first observe that the right-hand side of \eqref{eq:Aabs} is
  unambiguous.  The matrices $\Psi_t$ and $R_t$ depend on:
  \begin{itemize}
  \item a choice of ordering of the canonical co-ordinates
    $u^0,\dots,u^N$ at $t$; and
  \item the choice of square roots $\sqrt{\Delta^i(t)}$.
  \end{itemize}
  Thus any two different choices of $\Psi_tR_t$ are related by right
  multiplication by a signed permutation matrix.  Now $\cT$ is almost
  invariant under a signed permutation $(\bq^0,\dots,\bq^N) \mapsto
  (\pm \bq^{\sigma(0)},\dots,\pm \bq^{\sigma(N)})$: the only
  non-invariant part is the genus-one log-term $-\frac{1}{24}
  \sum_{\alpha} \log (-q_1^\alpha)$.  The constant ambiguity in this
  genus-one term cancels with the ambiguity of $-\frac{1}{48}\sum_i
  \log \Delta^i(t)$; the genus-one term $\cF^1_{\abs}$ in $\log \cAabs_t$ is
  normalized by the condition:
  \[
  \cF^1_{\abs}\big|_{\by(z)=0} = 0
  \]
  Thus $\cAabs_t$ is independent of all choices.

  Proposition~\ref{pro:Feynman} implies that $\widehat{\Psi_t}
  \widehat{R_t} (\cT)$ is a rational element of $\Fock\big(T_t\cM,
  \sum_{i=0}^N \sqrt{\Delta^i(t)} \parfrac{}{u^i})\big)$ with discriminant:
  \[
  \prod_{i=0}^N (-[\Psi_t^{-1} q_1]^i) 
  \]
  where $q_1 \in T_t\cM$. Because $\widehat{\Psi_t} \widehat{R_t}
  (\cT)$ is rational, $\cAabs_t = e^{{-\frac{1}{48}} \sum_i \log
    \Delta^i(t)}\, \widehat{\Psi_t} \widehat{R_t} (\cT)$ can naturally
  be regarded, via analytic continuation, as an element of
  $\Fock(T_t\cM,\unit)$. 
  One can normalize the discriminant by the non-zero factor 
  $e^{\frac{1}{2} \sum_i \log \Delta^i(t)}$ 
  \begin{align*}
    P(q_1) & = e^{\frac{1}{2} \sum_i \log\Delta^i(t)} 
    \prod_{i=0}^N (-[\Psi_t^{-1} q_1]^i) \\
    & = 
    \det( - q_1 *_t) 
  \end{align*}
  so that $P(-\unit) =1$. 
\end{proof}

\begin{remark} 
  \label{rem:absan-analytic} 
  When $t$ varies, $\cAabs_t$ defines a rational element of
  $\Fock(T\cM(U),\unit)$ with $U$ a neighbourhood of $t$. Here $T\cM(U)$
  is regarded as a free $\cO(U)$-module.
\end{remark} 

\begin{remark} \label{rem:Rfield}
  The transformation $R_t = I + R_1(t) z + R_2(t) z^2 + \cdots$ in
  Proposition~\ref{pro:R} can be determined by solving the equations:
  \[
  \Bigg( z \parfrac{}{z} + \frac{1}{z} \big({E
  \ast_t}\big) + \mu \Bigg) \Psi_t R_t \exp(U/z) = 0
  \]
  order by order in $z$.  It follows, and this will be important
  below, that if the canonical co-ordinates $u^i$ and the matrix entries
  of $\Psi_t$ all lie in some field of functions $k$, then the entries of
  each matrix $R_i(t)$ lie in $k$ too.
\end{remark}

\section{Givental's Formula in the Formal Setting}
\label{sec:Giventalformulaformal}
Note that the discussion in \S\ref{sec:Giventalformulaanalytic} makes
sense, and the analog of Proposition~\ref{pro:R} holds, in the setting
where $\cM$ is a formal Frobenius manifold over an algebraically
closed field $k$ of characteristic zero.  In this case $\cM$ is the
formal neighbourhood of zero in a vector space $H$, so $\cM = \Spf
k\pss$ where $\phi_0,\ldots,\phi_N$ is a basis for $H$ and $s = s^0
\phi_0 + \ldots + s^N \phi_N$ is a point of $H$.  The family of
products on the tangent spaces to $\cM$ give (and are given by) a
$k\pss$-bilinear product $\ast$ on $H\pss$.  We choose $\phi_0$ to be
the identity of the product $\ast$.  A formal Frobenius manifold is
said to be \emph{semisimple at the origin} if the algebra
$(H,{\ast|_{s=0}})$ is semisimple.  (The origin is in any case the
only $k$-valued point of $\cM$.)  Then, since $k$ is algebraically
closed, distinct eigenvalues $u^0,\ldots,u^N$ for $(E \ast)$ exist in
$k\pss$; these form canonical co-ordinates on a formal neighbourhood
of $s=0$ in $\cM$.  The vectors $\parfrac{}{u^i}$ are idempotents in
the algebra $\Big(H\pss,\ast\Big)$, and we define $\Delta^i \in k\pss$
by:
\[
\Delta^i = {1 \over
  g\Big(\parfrac{}{u^i},\parfrac{}{u^i}\Big)}
\]
For Proposition~\ref{pro:R}, we replace:
\begin{align*}
  & \Psi_u \in \Hom\big(\C^{N+1},T_u \cM\big) 
  && \text{by} && \Psi \in \Hom\big(k^{N+1},H\big)\pss \\
  & R_u \in \End\big(\C^{N+1}\big) \otimes \C[\![z]\!] && \text{by} &&
  R \in \End\big(k^{N+1}\big)\big[\!\big[z\big]\!\big]\pss
\end{align*}
with the rest of the conditions unchanged.  In other words: the
canonical co-ordinates $u^i$, the normalizations $\Delta^i$, and the
transformations $\Psi$ and $R$ are all defined in a formal
neighbourhood of $s=0$ in $\cM$.

\begin{proposition}[formal version of Proposition~\ref{pro:R}]
  \label{pro:formalR}
  The equation:
  \[
  \Bigg( z \parfrac{}{z} + \frac{1}{z} \big({E \ast}\big) + \mu \Bigg)
  S = 0
  \]
  has a unique solution of the form $S = \Psi R \exp(U/z)$ such that:
  \begin{enumerate}
  \item $\Psi \in \Hom\big(k^{N+1},H\big)\pss$ sends the $i$th
    standard basis vector in $k^{N+1}$ to the $i$th normalized
    idempotent $\sqrt{\Delta^i} \parfrac{}{u^i} \in H\pss$
  \item $R \in \End\big(k^{N+1}\big)\big[\!\big[z\big]\!\big]\pss$ with $R
    \equiv \Id \mod z$
  \item $U = \operatorname{diag}\big(u^0,\ldots,u^N\big)$
  \end{enumerate}
  The transformation $R$ satisfies\footnote{As in the analytic case,
    the transformation $R$ here is in addition automatically flat with
    respect to the Dubrovin connection and homogeneous with respect to
    the Euler vector field.}:
  \[
  R(-z)^{\mathrm{T}} R(z) = \Id
  \]
\end{proposition}

The composition $\Psi R \colon k^{N+1}\pss[\![z]\!] 
\to H\pss [\![z]\!]$ is a unitary isomorphism 
(see Definition \ref{def:unitarity}) 
over the ground ring $k\pss$, thus the 
following definition makes sense. 

\begin{definition}[formal version of Definition~\ref{def:ancestor}]
  \label{def:ancestorformal}
  The \emph{abstract ancestor potential} $\cAabs_s$ is:
  \[
  \cAabs_s = 
  \Big(\textstyle\prod_{i=0}^{i=N} \Delta^i\Big)^{-\frac{1}{48}}
  \, \widehat{\Psi} \widehat{R} (\cT)
  \]
\end{definition}

Just as in Proposition~\ref{pro:Aabs}, $\cAabs_s$ is a well-defined
rational element of $\Fock(H\pss, \phi_0)$ with 
discriminant $P(q_1) = \det(-q_1\ast) \in 
k\pss[q_1^0,\dots,q_1^N]$.

\section{Teleman Implies Givental}
\label{sec:TelemanimpliesGivental}
Let $X$ be a smooth projective toric variety.  Recall the definition
of the total ancestor potential $\cA_X$ in
equation~\ref{eq:ancestorpotential}.  The genus-zero Gromov--Witten
potential $F^0_X$ converges \cite{Iritani:convergence} in the sense of
\S\ref{sec:convergence}, and so the quantum cohomology of $X$ defines
an analytic Frobenius manifold (see
\S\ref{sec:Giventalformulaanalytic}).  This Frobenius manifold is
semisimple \cite{Iritani:convergence}.  When $X$ is a Fano toric
variety, Givental proves that:
\[
\cA_X|_{Q_1=\cdots=Q_r=1} = \cAabs
\]
by establishing a similar formula in the \emph{equivariant}
Gromov--Witten theory of $X$ and then taking a non-equivariant limit.
His argument simultaneously proves:
\begin{itemize}
\item[(A)] The convergence of $\cA_X|_{Q_1 = \cdots = Q_r=1}$, in the
  sense of Definition~\ref{def:ancestor-converge};
\item[(B)] The equality $\cA_X|_{Q_1 = \cdots = Q_r = 1} = \cAabs$,
  where the right-hand side is defined as in
  \S\ref{sec:Giventalformulaanalytic}.
\end{itemize}
Givental conjectured that (A) and (B) hold in general.  His
calculation in equivariant Gromov--Witten theory in fact applies to
any smooth projective toric variety $X$, and Iritani has proven that
one can take the non-equivariant limit of this calculation even if $X$
is not Fano \cite{Iritani:convergence}, so (A) and (B) are known to
hold whenever $X$ is a smooth projective toric variety.

In this section we explain how Givental's statements (A) and (B) can
be deduced in much greater generality from Teleman's classification of
Deligne--Mumford Field Theories (DMTs) \cite{Teleman}.  Teleman proves
\cite{Teleman}*{Theorem~1} that if a DMT satisfies:
\begin{itemize} 
\item a \emph{Cohomological Field Theory} condition;
\item a \emph{homogeneity} condition (involving an Euler vector field);
\item a \emph{flat vacuum} condition (involving the identity element of the
  Frobenius algebra);
\end{itemize} 
and if its genus-zero part defines a semisimple Frobenius algebra,
then:
\begin{itemize}
\item the DMT can be uniquely reconstructed from its genus-zero part;
  and
\item the ancestor potential of the DMT coincides with Givental's
  abstract potential $\cAabs$.  
\end{itemize}
Teleman's argument works over an arbitrary field of characteristic zero.

We now consider three conditions on the Gromov--Witten invariants of a
projective variety $X$.  Let $k$ denote the algebraic closure of the
fraction field of $\Lambda[\![t]\!]$.  The first condition, which we
call Formal Semisimplicity, is:
\begin{equation}
  \label{eq:formalsemisimplicity}
  \text{the quantum cohomology algebra $\big(H_X \otimes k, {\ast} \big)$ is semisimple}
\end{equation}
The second condition, which we call Genus-Zero Convergence, is:
\begin{equation}
  \label{eq:genuszeroconvergence}
  \text{the genus-zero Gromov--Witten potential $F^0_X$ converges
    in the sense of \S\ref{sec:convergence}}
\end{equation}
Let $\cM\subset H_X\otimes \C$ be a neighbourhood \eqref{eq:LRLnbhd}
of the large-radius limit point.  If Genus-Zero Convergence holds
then, as discussed in \S\ref{sec:Giventalformulaanalytic}, the
genus-zero Gromov--Witten theory of $X$ defines on $\cM$ the structure
of an analytic Frobenius manifold over $\C$.  The third condition,
which we call Analytic Semisimplicity, is:
\begin{equation}
  \label{eq:analyticsemisimplicity}
  \text{this analytic Frobenius manifold is generically semisimple}
\end{equation}

\begin{remark}
  Formal Semisimplicity \eqref{eq:formalsemisimplicity} and Genus-Zero
  Convergence \eqref{eq:genuszeroconvergence} together imply Analytic
  Semisimplicity \eqref{eq:analyticsemisimplicity}, and Genus-Zero
  Convergence \eqref{eq:genuszeroconvergence} and Analytic
  Semisimplicity \eqref{eq:analyticsemisimplicity} together imply
  Formal Semisimplicity \eqref{eq:formalsemisimplicity}.
\end{remark}

\begin{remark}
  All three conditions are satisfied when $X$ is a smooth projective
  toric variety: this follows from mirror symmetry for toric varieties
  \citelist{\cite{Givental:toric} \cite{Hori--Vafa}
    \cite{Iritani:convergence}}.
\end{remark}

\begin{remark}
  If both Genus-Zero Convergence \eqref{eq:genuszeroconvergence} and
  Analytic Semisimplicity \eqref{eq:analyticsemisimplicity} hold then
  we can define the abstract ancestor potential $\cAabsan$ as in
  \S\ref{sec:Giventalformulaanalytic}.  The subscript `$\mathrm{an}$'
  here is to emphasize that we are working in the analytic setting.
\end{remark}

In \S\ref{sec:telemanformal} below we show that if Formal
Semisimplicity holds then we can apply Teleman's theorem to the
Gromov--Witten theory of $X$, thereby recovering the total ancestor
potential $\cA_X$ from the quantum cohomology.  In
\S\ref{sec:ancestorconvergence} we show that if both Genus-Zero
Convergence and Analytic Semisimplicity hold then the total ancestor
potential $\cA_X$ is convergent in the sense of Definition
\ref{def:ancestor-converge}, and is equal to the abstract ancestor
potential $\cAabsan$.

\subsection{Applying Teleman's Theorem in the Formal Setting}
\label{sec:telemanformal}

Recall that $k$ denotes the algebraic closure of the fraction field of
$\Lambda[\![t]\!]$.  The quantum cohomology $(H_X\otimes k, *)$ over
$k$ is equipped with the element:
\begin{equation} 
\label{eq:Eulerelement}
E = t^0 \phi_0 + c_1(X) + \sum_{i=r+1}^N 
\left(1-\frac{1}{2} \deg \phi_i\right) t^i \phi_i  
\end{equation} 
corresponding to the Euler vector field \eqref{eq:Eulerfield}.  If
Formal Semisimplicity \eqref{eq:formalsemisimplicity} holds, then we
have the decomposition:
\begin{align*}
  H_X \otimes k = \bigoplus_{i=1}^N k \delta_i,
  &&
  \delta_i * \delta_j =
  \begin{cases}
    \delta_i & \text{if $i=j$} \\
    0 & \text{otherwise}
  \end{cases}
\end{align*}
and $(E*)$ is a semisimple operator with eigenvalues $u^0,\dots,u^N\in
k$ such that $E* \delta_i = u_i \delta_i$.  We define $\Delta^i\in k$
by
\[
\Delta^i = \frac{1}{g(\delta_i,\delta_i)}
\]
where $g$ stands for the Poincar\'{e} pairing.  Then, as in
Proposition \ref{pro:formalR}, the differential equation:
\[
\left(z\parfrac{}{z} + \frac{1}{z} E* + \mu \right) S = 0 
\]
has a unique solution of the form $S= \Psi R e^{U/z}$ such that:
\begin{enumerate} 
\item $\Psi \in \Hom(k^{N+1},H_X\otimes k)$ sends the 
$i$th standard basis vector in $k^{N+1}$ to the $i$th 
normalized idempotent $\sqrt{\Delta^i} \delta_i$.  
\item $R\in \End(k^{N+1},k^{N+1})[\![z]\!]$ with $R\equiv \Id \mod z$. 
\item $U = \diag(u^0,\dots,u^N)$. 
\end{enumerate} 
Hence we can define the abstract ancestor potential as:
\[
\cAabsformal = e^{{-\frac{1}{48}} \sum_i \log \Delta^i}\,
\widehat{\Psi} \widehat{R} (\cT).
\]
(cf.~Definitions~\ref{def:ancestor} and~\ref{def:ancestorformal}).
$\cAabsformal$ is a rational element of $\Fock(H_X\otimes k, \phi_0)$
with discriminant $\det(-q_1*)$.  
We will see below that it arises from a formal Frobenius 
manifold over $k$ as the ancestor potential at the origin.  

\begin{theorem}[Teleman \cite{Teleman}]
\label{thm:Ageom=Aabs}
Let $X$ be a smooth projective variety such that Formal Semisimplicity
\eqref{eq:formalsemisimplicity} holds.  Recall the definition of the
total ancestor potential $\cA_X$ in
equation~\ref{eq:ancestorpotential}, and the definition of the ring
$\cS$ in Remark~\ref{rem:ancestordivisor}.  We have:
\[
\cA_X = \cAabsformal.
\] 
In particular $\cA_X$ is a rational element of $\Fock(H_X\otimes
\cS,\phi_0)$, with discriminant
$\det(-q_1*)$.
\end{theorem} 

\begin{proof}
  This is a direct consequence of Teleman's result.  We spell out how
  the Gromov--Witten theory of $X$ defines both a Deligne--Mumford
  Field Theory (DMT) over $k$ and a formal Frobenius manifold over
  $k$.  This formal Frobenius manifold induces at the origin the data
  defined above: the Frobenius algebra $(H_X\otimes k, *, g)$ together
  with $E$ and $\mu$.

  \subsection*{Step 1: A DMT over $k$.} We first make minor
  adjustments to the formal setup in Teleman \cite{Teleman}.  Recall
  that a DMT is a family of maps:
  \begin{align*}
    Z_g^n \colon & H_X^{\otimes n} \longrightarrow H^\bullet(\cMbar_{g,n})
    && 2g-2+n>0
  \end{align*}
  satisfying certain \emph{factorization axioms} and a \emph{vacuum
    axiom}.  Pulling back cohomology classes along the maps
  $\ev_i:X_{g,n,d} \to X$, capping with the virtual fundamental class,
  and then pushing forward along the canonical map $X_{g,n,d} \to
  \cMbar_{g,n}$ defines maps:
  \begin{align*}
    GW_{g,d}^n \colon & H_X^{\otimes n} \longrightarrow
    H^\bullet(\cMbar_{g,n}) && 2g-2+n>0
  \end{align*}
  and setting:
  \[
  Z_g^n = \sum_{d \in \NE(X)} GW_{g,d}^n \, Q^d
  \]
  defines a DMT over $\Lambda$.  Let $t \in H_X$ be $t = t^0 \phi_0 +
  \cdots + t^N \phi_N$ as before.  Setting:
  \begin{align*}
    \leftsub{t}{Z}_g^n(x_1,\ldots,x_n) = \sum_{m \geq 0} {1 \over m!}
    \int_{\cMbar_{g,n+m}}^{\cMbar_{g,n}} Z_g^{n+m} (x_1,\ldots,x_n,t,\ldots,t)
    && 2g-2+n>0
  \end{align*}
  where the integral denotes the push-forward along the canonical map
  $\cMbar_{g,n+m} \to \cMbar_{g,n}$, defines a formal family of DMTs
  over $\Lambda$, parametrized by $\Spf \Lambda[\![t]\!]$; cf
  \cite{Teleman}*{\S7}.  We regard this as a single DMT over the field
  $k$.
  
  \subsection*{Step 2: A formal Frobenius manifold over $k$.}  We now
  deform this DMT to construct a family of DMTs parametrized by $\Spf
  k \big[\!\big[s^0,\ldots,s^N \big]\!\big]$, and hence a formal
  Frobenius manifold over $k$.  (The genus-zero part of any DMT is a
  tree-level Cohomological Field Theory in the sense of
  \cite{Manin}*{III.4}, and thus determines a formal Frobenius
  manifold.)\phantom{.}  Define:
  \begin{align*}
    \leftsub{s,t}{Z}_g^n(x_1,\ldots,x_n) = \sum_{m \geq 0} {1 \over m!}
    \int_{\cMbar_{g,n+m}}^{\cMbar_{g,n}} \leftsub{t}{Z}_g^{n+m} (x_1,\ldots,x_n,s,\ldots,s)
    && 2g-2+n>0
  \end{align*}
  where $s \in H_X$ is $s = s^0 \phi_0 + \cdots + s^N \phi_N$.  As in
  \cite{Teleman}*{\S7}, this defines a family of DMTs over $k$,
  parametrized by $\Spf k \big[\!\big[s^0,\ldots,s^N \big]\!\big]$.
  It is easy to check that this family is homogeneous\footnote{See
    \cite{Teleman}*{Definition~7.16}.} of weight $D = \dim_\C X$ with
  respect to the Euler vector field $\cE$ on $\Spf k\pss$:
  \[
  \cE = \rho^1 \parfrac{}{s^1} + \cdots + \rho^r \parfrac{}{s^r} +
  \sum_{i=0}^{i=N} \Big( 1 - \textstyle\frac{\deg \phi_i}{2}  \Big) 
  (s^i +t^i) \displaystyle\parfrac{}{s^i} 
  \]
  where $c_1(X) = \rho^1 \phi_1 + \cdots + \rho^r \phi_r$; note the
  shift compared to the Euler field in equation \eqref{eq:Eulerfield}.
  The formal Frobenius manifold over $k$ defined by the DMT is
  therefore conformal with Euler vector field $\cE$.
 The Euler vector field $\cE$ induces the element
  \eqref{eq:Eulerelement} at the origin and defines the grading
  operator $\mu$ by:
  \[
  \mu = \big(1 - \textstyle\frac{D}{2}\big)\Id - \nabla^{\LC} \cE. 
  \]
  Formal Semisimplicity \eqref{eq:formalsemisimplicity} guarantees
  that this formal Frobenius manifold induces a semisimple Frobenius
  algebra $(H_X\otimes k, *,g)$ at the origin.
  
  \subsection* {Step 3: Applying Teleman's Theorem.}
  Teleman's Theorem now implies that the ancestor potential for the
  family of DMTs constructed in Step 2 coincides with the abstract
  ancestor potential for the formal Frobenius manifold constructed in
  Step 2.  On setting $s = 0$, the ancestor potential for the family
  of DMTs becomes the geometrically-defined ancestor potential
  $\cA_X$ (see equation~\ref{eq:ancestorpotential}).  Thus:
  \[
  \cA_X = \cAabsformal.
  \]
  The right-hand side here is, \emph{a priori}, a formal power series
  in the variables $y^\beta_j$ with coefficients in $k$, but since it
  coincides with the left-hand side we know from
  Remark~\ref{rem:ancestordivisor} that it is in fact a formal power
  series in the variables $y^\beta_j$ with coefficients in $\cS$.
  Moreover, $\cAabsformal$ is rational over $k$ with 
discriminant $\det(-q_1*)$; this implies that
  $\cA_X$ is rational over $\cS$ with 
  discriminant $\det(-q_1*)$.
\end{proof}

\subsection{Convergence of the Total Ancestor Potential}
\label{sec:ancestorconvergence}

\begin{theorem} 
  \label{thm:convergence-ancestor} 
  Let $X$ be a smooth projective variety such that Genus-Zero
  Convergence \eqref{eq:genuszeroconvergence} and Analytic
  Semisimplicity \eqref{eq:analyticsemisimplicity} hold.  The total
  ancestor potential $\cA_X$ is convergent in the sense of Definition
  \ref{def:ancestor-converge}; more precisely $\cA_X$ is a rational
  element of $\Fock(H_X\otimes \cS_\epsilon,\phi_0)$, for some
  $\epsilon>0$, with discriminant $\det(-q_1*)$.
  Moreover we have:
  \[
  \cA_X|_{Q_1 = \cdots = Q_r = 1} = \cAabsan. 
  \]
\end{theorem} 

\begin{proof} 
  Let $\FractionField$ denote the fraction field and overline denote
  the algebraic closure, so that:
  \[
  k= \ov{\FractionField \Lambda[\![t]\!]}
  \]
  Let:
  \begin{align*} 
    & k_1 = \ov{\FractionField \Q\big[\!\big[t^0,Q_1 e^{t^1}, \dots,Q_r
      e^{t^r}, t^{r+1}, \dots, t^N\big]\!\big] } \\
    & k_2 = \ov{\FractionField \Q\big[\!\big[t^0,e^{t^1},\dots,e^{t^r},
      t^{r+1},\dots, t^N\big]\!\big]} \\
    & k_3 = \ov{\FractionField \Q\big\{ 
      t^0,e^{t^1},\dots, e^{t^r},t^{r+1},\dots,t^N \big\} } \\
    & k_4 = \Q\big[\!\big[t^0,e^{t^1},\dots,e^{t^r},t^{r+1},\dots,t^N\big]\!\big] \\
    & k_5 =  \Q\big\{t^0,e^{t^1},\dots,e^{t^r},t^{r+1},\dots,t^N\big\} 
  \end{align*}
  Lemma~\ref{lem:k3k4k5} below shows that $k_3 \cap k_4 = k_5$.

  The Divisor Equation implies that all of the ingredients $\Delta^i$,
  $\Psi$, and $R$ used to define $\cAabsformal$ (in
  \S\ref{sec:telemanformal}) are defined over $k_1$, and therefore
  that $\cAabsformal$ is an element of $\Fock(H_X\otimes k_1,
  \phi_0)$.  The specialization $Q_1= \cdots =Q_r=1$ defines an
  isomorphism $k_1 \cong k_2$, and thus $\cAabsformal|_{Q_1= \cdots =
    Q_r=1}$ is a well-defined element of $\Fock(H_X\otimes
  k_2,\phi_0)$.

  On the other hand all of the ingredients $\Delta^i(t)$, $\Psi_t$,
  and $R_t$ used to define $\cAabsan$ (in
  \S\ref{sec:Giventalformulaanalytic}) are defined over $k_3$, and
  therefore $\cAabsan$ is an element of $\Fock(H_X\otimes
  k_3,\phi_0)$.  Note that $k_3$ is contained in $k_2$.  Because the
  two sets of ingredients $(\Delta^i(t), \Psi_t,R_t)$ and $(\Delta^i,
  \Psi,R)$ coincide under the maps between ground fields $k_3 \to k_2$
  and $k_1 \to k_2$, it follows that
  \begin{equation} 
    \label{eq:absan=absformal} 
    \cAabsan = \cAabsformal|_{Q_1=\cdots = Q_r=1} 
  \end{equation} 
  as elements of $\Fock(H_X\otimes k_2,\phi_0)$. 

  By Theorem \ref{thm:Ageom=Aabs}, the right-hand side of
  \eqref{eq:absan=absformal} equals $\cA_X|_{Q_1= \cdots =Q_r=1}$ and
  is an element of $\Fock(H_X\otimes k_4, \phi_0)$.  Note that $k_4$
  is contained in $k_2$.  Since the left-hand side of
  \eqref{eq:absan=absformal} is defined over $k_3\subset k_2$, it
  follows that $\cAabsan$, $\cAabsformal|_{Q_1= \cdots =Q_r =1}$, and
  $\cA_X|_{Q_1= \cdots =Q_r=1}$ (which are all equal) are all defined
  over $k_3 \cap k_4 = k_5$, i.e.~all three are elements of
  $\Fock(H_X\otimes k_5,\phi_0)$.  

  Because $\cM$ is a neighbourhood \eqref{eq:LRLnbhd} of the
  large-radius limit point, it contains the set
  \[
  \{(t^0,\dots,t^N)\; |\; (t^0,e^{t^1},\dots,e^{t^r},t^{r+1},\dots,t^N) 
  \in B_\epsilon \}
  \]
  for some $\epsilon>0$, where $B_\epsilon = \{(z_0,\dots,z_N) \in
  \C^{N+1}\;|\; |z_i|< \epsilon \}$.  By Remark
  \ref{rem:absan-analytic}, $\cAabsan$ is also an element of
  $\Fock(H_X\otimes \cO(B_\epsilon^{\rm ss}), \phi_0)$ where
  $B_\epsilon^{\rm ss}\subset B_\epsilon\cap (\C\times
  (\C^\times)^r\times \C^{N-r})$ denotes the semisimple locus.
  Therefore, when expanding $\log \cAabsan$ in variables $y_j^\beta$
  and $\hbar$, each coefficient is analytic function on
  $B_\epsilon^{\rm ss}$ which extends to a neighbourhood of the origin
  in $B_\epsilon$.  Observe that $B_\epsilon^{\rm ss}$ is an analytic
  Zariski open subset in $B_\epsilon$ and that $Z=B_\epsilon \setminus
  B_\epsilon^{\rm ss}$ is a locally finite union of irreducible
  analytic subvarieties.  Thus there exists $\epsilon'$ such that
  $0<\epsilon'<\epsilon$ and that $B_{\epsilon'}$ does not meet any
  irreducible component of $Z$ which is away from the origin.  Every
  coefficient (of the expansion of $\log \cAabsan$ in variables
  $y_j^\beta$ and $\hbar$) extends to a holomorphic function on
  $B_{\epsilon'}$.  This shows that $\cA_X|_{Q_1= \cdots =Q_r=1}$ is convergent in the
  sense of Definition \ref{def:ancestor-converge}, or in other words:
  \[
  \cA_X|_{Q_1= \cdots =Q_r=1}
  \in \Fock(H_X \otimes \cS_{\epsilon'}, \phi_0)
  \]
  Finally, the rationality of $\cA_X|_{Q_1= \cdots =Q_r=1}$ follows
  from the rationality of $\cAabsan$ and the fact that the
  discriminant $\det(-q_1*)$ is an element of
  $\cS_{\epsilon'}[q_1^0,\dots,q_1^N]$.
\end{proof} 

\begin{lemma} 
  \label{lem:k3k4k5}
 Consider the intersections:
  \begin{align*}
    & \ov{\FractionField \C\{x_1,\ldots,x_n\}} \cap \C[\![x_1,\ldots,x_n]\!] \subset
    \ov{\FractionField \C[\![x_1,\ldots,x_n]\!]} \\
    & \ov{\FractionField \Q\{x_1,\ldots,x_n\}} \cap \Q[\![x_1,\ldots,x_n]\!] \subset
    \ov{\FractionField \Q[\![x_1,\ldots,x_n]\!]} 
 \end{align*}
 We have:
 \begin{enumerate}
  \item $\ov{\FractionField \C\{x_1,\ldots,x_n\}} \cap \C[\![x_1,\ldots,x_n]\!]  =
   \C\{x_1,\ldots,x_n\}$\\[-1.9ex]
 \item $\ov{\FractionField \Q\{x_1,\ldots,x_n\}} \cap \Q[\![x_1,\ldots,x_n]\!]  =
   \Q\{x_1,\ldots,x_n\}$
 \end{enumerate}
\end{lemma}
\begin{proof}
  Statement~(1) immediately implies statement~(2).  We prove (1).  Let:
  \[
  P(x_1,\ldots,x_n,y) = f_0(x_1,\ldots,x_n) y^k + f_1(x_1,\ldots,x_n) y^{k-1} +
  \cdots + f_k(x_1,\ldots,x_n)
  \]
  where $f_i \in \C\{x_1,\ldots,x_n\}$. Assume that the equation
  $P(x_1,\ldots,x_n,y) = 0$ has a solution $y = g(x_1,\ldots,x_n)$
  with $g \in \C[\![x_1,\ldots,x_n]\!]$, so that:
  \[
  P\big(x_1,\ldots,x_n,g(x_1,\ldots,x_n)\big) = 0
  \] 
  We will show that $g \in \C\{x_1,\ldots,x_n\}$.  Without loss of
  generality we may assume that $g(0,\ldots,0) = 0$, and therefore
  that $P(0,0,\ldots,0,0) = 0$.

  Suppose first that $P(0,0,\ldots,0,y)$ is not identically zero.
  Then the Weierstrass preparation theorem implies that:
  \[
  P(x_1,\ldots,x_n,y) =   W(x_1,\ldots,x_n,y)\, h(x_1,\ldots,x_n,y) 
  \]
  where $h$ is a unit in the local ring at the origin and $W$ is a
  Weierstrass polynomial:
  \[
  W(x_1,\ldots,x_n,y) = y^l + \sum_{j=0}^{l-1} w_j(x_1,\ldots,x_n)
  y^{j}
  \]
  with $w_j(0,\ldots,0) = 0$.  Then
  $W(x_1,\ldots,x_n,g(x_1,\ldots,x_n)) = 0$.  A theorem of Aroca 
  \cite{Aroca} 
  implies that there exist vectors:
  \begin{align*}
    & v_1,\ldots,v_n \in \Q^n && v_i = (v_i^1,\ldots,v_i^n)
  \end{align*}
  such that $v_1,\ldots,v_n$ span a strictly convex cone containing
  the positive orthant, that the $\Z_{\geq 0}$-span of
  $v_1,\ldots,v_n$ contains $\big(\Z_{\geq 0}\big)^n$, and that after
    the monomial change of variables:
  \begin{align*}
    & z_i = x_1^{v_i^1}\cdots x_n^{v_i^n} && i = 1, 2, \ldots, n
  \end{align*}
  there exists a convergent power series $y_c \in
  \C\{z_1,\ldots,z_n\}$ such that:
  \[
  W\big(x_1,\ldots,x_n,y_c(z_1,\ldots,z_n)\big) = 0
  \]
  One can therefore factorize $W$ over the ring $\C\{z_1,\ldots,z_n\}$:
  \[
  W(x_1,\ldots,x_n,y) = (y-y_c)\left(y^{l-1} + \sum_{j=0}^{l-2}
    w_j'(z_1,\ldots,z_n) y^j \right)
  \]
  This equation makes sense over the ring $\C[\![z_1,\ldots,z_n]\!]$
  which contains the solution $y = g(x_1,\ldots,x_n)$.  Thus either
  $y_c = g$, in which case $g\in \C\{x_1,\ldots,x_n\}$, or we can apply Aroca's
  theorem again with $W(x_1,\ldots,x_n,y)$ replaced by the Weierstrass
  polynomial:
  \[
  y^{l-1} + \sum_{j=0}^{l-2} w_j'(z_1,\ldots,z_n) y^j
  \]
  of lower degree.  By induction, we conclude that $g \in \C\{x_1,\ldots,x_n\}$.

  It remains to consider the case where $P(0,0,\ldots,0,y)$ is
  identically zero.  Consider the co-ordinate change:
  \begin{align*}
    &x_i' = x_i - a_i y && 1 \leq i \leq n
  \end{align*}
  where we choose $(a_1,\ldots,a_n) \in \C^n$ such that
  $P(x_1,\ldots,x_n,y)$ is not identically zero on the line $x_1' =
  \ldots = x_n' = 0$, and that $dg_{(0,0\ldots,0)}(a_1,\ldots,a_n) \ne
  1$.  Writing the solution $y = g(x_1,\ldots,x_n)$ in the new
  co-ordinate system, we find:
  \[
  y = g(x'_1+a_1 y, x'_2+a_2 y, \ldots,x'_n + a_n y)
  \]
  This equation has a unique power series solution $y =
  G(x_1',\ldots,x_n')$, and the argument in the preceding paragraph
  shows that $G \in \C\{x_1',\ldots,x_n'\}$.  To recover
  $g(x_1,\ldots,x_n)$ from $G(x_1',\ldots,x_n')$ we solve the
  equation:
  \[
  y = G(x_1 - a_1 y, x_2 - a_2 y,\ldots, x_n - a_n y)
  \]
  This too has a unique power series solution $y = g(x_1,\ldots,x_n)$,
  because the condition $dg_{(0,0\ldots,0)}(a_1,\ldots,a_n) \ne 1$
  implies that $dG_{(0,0\ldots,0)}(a_1,\ldots,a_n) \ne {-1}$.  On the
  other hand, the implicit function theorem shows that there is a
  unique analytic solution $y = v(x_1,\ldots,x_n)$ such that
  $v(0,\ldots,0) = 0$.  The power series expansion of $v$ at the
  origin must coincide with $g(x_1,\ldots,x_n)$; thus $g \in
  \C\{x_1,\ldots,x_n\}$.  The Lemma is proved.
\end{proof}

\begin{remark} 
  The same argument proves Givental's statements~(A) and~(B) for the
  ancestor potential of a compact toric \emph{orbifold}.  We need:
  \begin{itemize}
  \item the fact that orbifold Gromov--Witten theory defines a DMT
    (combine \cite{Teleman}*{\S1.7} with \cite{AGV})
  \item analyticity, semisimplicity, and tameness of the corresponding
    Frobenius manifold.
  \end{itemize}
  This last point would follow from an appropriate mirror theorem for
  toric orbifolds.  Such a mirror theorem has been formulated as a
  conjecture by Coates--Corti--Iritani--Tseng (see
  \cite{Iritani:integral}*{\S4}), proved for weighted projective
  spaces in \cite{CCLT}, and will be proved for general toric
  orbifolds $X$ in \cite{CCIT:forthcoming}.

  Tseng has announced a proof of statements (A) and (B) for compact
  toric orbifolds using localization in equivariant Gromov--Witten
  theory \cite{Milanov--Tseng}.  His version is somewhat stronger than
  ours, as it applies in the equivariant setting where the Frobenius
  manifold is not conformal.
\end{remark}

\section{NF-Convergence of Gromov--Witten Potentials: Statements}
\label{sec:statements}

\begin{definition} 
  \label{def:Frechet-convergence-ancestor}
  The genus-$g$ ancestor potential $\bar\cF^g_t$ is said to be
  \emph{NF-convergent} if the power series \eqref{eq:ancestor}
  converges absolutely and uniformly on an infinite-dimensional
  polydisc of the form:
  \begin{equation}
    \label{eq:formofpolydisc}
    \begin{cases}
      |y_i^\alpha|<  \epsilon  \frac{i!}{C^i}      
      & \text{$0 \leq i < \infty$,  $0 \leq \alpha \leq N$} \\
      |t^\alpha| < \epsilon 
      & \text{$0 \leq \alpha \leq N$} \\
      |Q_j| < \epsilon 
      & \text{$1 \leq j \leq r $}
    \end{cases}
  \end{equation}
  for some $C,\epsilon>0$.  The total ancestor potential $\cA_X$ is
  said to be \emph{NF-convergent} if the power series
  \eqref{eq:ancestor} defining each genus-$g$ ancestor potential
  $\bar\cF^g_t$ converges absolutely and uniformly on a polydisc of
  the form \eqref{eq:formofpolydisc} for some uniform $C,\epsilon>0$.
\end{definition}  

\begin{remark}
  ``NF'' here stands for ``nuclear Fr\'echet'': see
  Remark~\ref{rem:holomorphy} below.
\end{remark}

\begin{theorem} 
  \label{thm:convergenceimpliesFrechetconvergenceancestor}
  If the total ancestor potential $\cA_X$ is convergent in the sense
  of Definition \ref{def:ancestor-converge}, then it is NF-convergent
  in the sense of Definition \ref{def:Frechet-convergence-ancestor}.
\end{theorem} 

\begin{remark} 
  NF-convergence of the total ancestor potential (Definition
  \ref{def:Frechet-convergence-ancestor}) is weaker than convergence
  of the total ancestor potential (Definition
  \ref{def:ancestor-converge}).  The rationality and the tameness in
  Definition \ref{def:ancestor-converge} do not follow from
  NF-convergence.
\end{remark} 

Theorem \ref{thm:convergence-ancestor} and
Theorem~\ref{thm:convergenceimpliesFrechetconvergenceancestor}
together immediately imply Theorem~\ref{thm:maintheoremancestors}.

\subsection{Convergence of the Descendant Potential}

\begin{definition}
  \label{def:descendant-converge} 
  The genus-$g$ descendant Gromov--Witten potential $\cF^g_X$ is said
  to be \emph{NF-convergent} if the power series
  \eqref{eq:genus_g_descendantpot} converges absolutely and uniformly
  on an infinite-dimensional polydisc of the form:
  \begin{equation} 
    \label{eq:convergence-domain-cFg} 
    \begin{cases}
      |t_i^\alpha|<  \epsilon  \frac{i!}{C^i}      
      & \text{$0 \leq i < \infty$, $0 \leq \alpha \leq N$} \\
      |Q_j| < \epsilon & 1 \leq j \leq r
    \end{cases}  
  \end{equation} 
  for some $C, \epsilon>0$.  We say that the total descendant
  Gromov--Witten potential $\cZ_X$ is \emph{NF-convergent} if
  the power series \eqref{eq:genus_g_descendantpot} defining each
  genus-$g$ descendant potential $\cF^g_X$ converges absolutely and
  uniformly on a polydisc of the form
  \eqref{eq:convergence-domain-cFg} for some uniform $C, \epsilon>0$.
\end{definition}

\begin{remark} 
  \label{rem:holomorphy} 
  A holomorphic function on a locally convex topological vector space
  over $\C$ can be defined as a complex G\^{a}teaux-differentiable
  function which is continuous \citelist{\cite{Dineen}\cite{Chae}}.  If $\cF_X^g$ is
  NF-convergent then it defines a holomorphic function on an
  $\epsilon$-ball of the Banach space:
  \begin{equation} 
    \label{eq:l_infinityC}
    l_{\infty}^C(H_X) = 
    \left\{ \bt(z) \in H_X \otimes \C[\![z]\!] : \sup_{i,\alpha} 
      \left( 
        \frac{|t_i^\alpha| C^{i}}{i!}  \right) < \infty \right \} 
  \end{equation} 
  equipped with the weighted $l_\infty$-norm:
  \begin{equation} 
  \label{eq:weighted-linfty}
  \| \bt \|_{\infty,\log C} = \sup_{i,\alpha}\left( 
    \frac{|t_i^\alpha| C^{i}}{i!}  \right)
  \end{equation} 
  If $\cF_X^g$ is NF-convergent then we can also view it as a
  holomorphic function on a neighbourhood of the origin of the nuclear
  Fr\'{e}chet space:
  \begin{equation} 
    \label{eq:H+} 
    \cH_+= 
    \left\{ \bt(z) \in H_X \otimes \C[\![z]\!] :
      \sup_{i,\alpha} \left(\frac{|t_i^\alpha| e^{in}}{i!}\right) < \infty
      \text{ for all $n \geq 0$}
    \right\} \subset l_{\infty}^C(H_X). 
  \end{equation} 
  The topology on $\cH_+$ is defined by countably many norms:
  \begin{align*}
    \|\bt\|_{\infty,n} = \sup_{i,\alpha} \left(\frac{|t_i^\alpha| e^{ni}}{i!}\right)
    && n=0,1,2,\dots
  \end{align*}
  This viewpoint is perhaps more natural.  
As we will see in Lemma~\ref{lem:Frechet-Banach}, 
a holomorphic function on a neighbourhood
  of zero in $\cH_+$ automatically extends to a holomorphic function
  on a neighbourhood of zero in $l_{\infty}^C(H_X)$ for \emph{some}
  $C>0$.
\end{remark} 

\begin{remark} 
  In unpublished work, Iritani has shown that the Gromov--Witten
  potential $\cF^g_X$ converges on a polydisc of the form
  \eqref{eq:convergence-domain-cFg} whenever the target space $X$
  admits a torus action with isolated fixed points and isolated
  1-dimensional orbits \cite{Iritani:localization}.
\end{remark} 

\begin{theorem} 
  \label{thm:convergence-genuszero} 
  If the non-descendant genus-zero potential $F^0_X$ is convergent in
  the sense of \S\ref{sec:convergence} then the genus-zero descendant
  potential $\cF^0_X$ is NF-convergent in the sense of
  Definition~\ref{def:descendant-converge}.
\end{theorem}

\begin{theorem} 
  \label{thm:descendant-convergence} 
  If the total ancestor potential $\cA_X$ is convergent in the sense
  of Definition~\ref{def:ancestor-converge} then the total descendant
  potential $\cZ_X$ is NF-convergent in the sense of
  Definition~\ref{def:descendant-converge}
\end{theorem}

Theorem~\ref{thm:convergence-ancestor} and
Theorem~\ref{thm:descendant-convergence} together immediately imply
Theorem~\ref{thm:maintheorem}.

\section{NF-Convergence of Gromov--Witten Potentials: Proofs}
\label{sec:proofs}

In this section we prove the results about NF-convergence of
descendant and ancestor potentials stated in \S\ref{sec:statements}.
The key ingredients are the Kontsevich--Manin ancestor-descendant
relation, the Nash--Moser inverse function theorem, and a version of
Givental's symplectic space based on a nuclear Fr\'echet space
(see \S\ref{sec:Giventalnuclear}) which may be of
independent interest.

\subsection{Setting $Q_1 = \cdots = Q_r = 1$ makes sense when
  $\cF^g_X$ is NF-convergent}
\label{sec:specializationmakesense}

Making the argument explicit, we write the genus-$g$ descendant 
potential $\cF^g_X$ as 
\[
\cF^g_X(\bq, Q_1,\dots,Q_r) 
\]
where $\bq$ is the dilaton-shifted co-ordinate appearing in \S
\ref{sec:dilaton} and $Q_1,\dots,Q_r$ are Novikov variables.  
The Divisor Equation \cite{AGV}*{Theorem~8.3.1} implies that:
\begin{align}
\label{eq:divisorequation-cFg} 
\begin{split} 
\cF^g_X([e^{-\delta/z} \bq(z)]_+, Q_1,\dots,Q_r) & =  
\cF^g_X(\bq(z), e^{\delta_1} Q_1,\dots,e^{\delta_r} Q_r) \\ 
& +  
\frac{\delta_{g,0}}{2} \Omega(e^{-\delta/z} \bq(z), [e^{-\delta/z} \bq(z)]_+) 
- 
\frac{\delta_{g,1}}{24} \int_X \delta \cup c_{D-1}(X) 
\end{split} 
\end{align} 
where $\delta = \sum_{\alpha=1}^r \delta_\alpha \phi_\alpha \in
H^2(X)$, $D=\dim X$, $[\cdots]_+$ denotes the power series truncation of a
Laurent series in $z$ and $\Omega$ is Givental's symplectic form 
in \eqref{eq:symplecticform} (with $V = H_X$). 
The formula follows by integrating \cite{Coates-Givental}*{Equation (8)} 
and using \cite{Givental:quantization}*{Proposition 5.3}. 
Equation \eqref{eq:divisorequation-cFg} is an
equality between formal power series in the variables $t_i^\alpha$,
$Q_i$ and $\delta_\alpha$, where:
\[
t_i^\alpha = 
\begin{cases}
  q_i^\alpha + 1 & \text{if $(i,\alpha)=(1,0)$} \\
  q_i^\alpha & \text{otherwise}
\end{cases}
\]
Note that the specialization $Q_1=\cdots =Q_r=1$ of the right-hand
side of \eqref{eq:divisorequation-cFg} makes sense as analytic function 
on a region $\{(\bq(z) = \bt(z) - \phi_0 z, \delta) : 
\|\bt\|_{\infty, \log C} < \epsilon, |e^{\delta_\alpha}| <\epsilon\}
\subset l^C_\infty(H_X)\times H^2(X;\C)$ 
if $\cF_X^g$ is NF-convergent 
(see \eqref{eq:l_infinityC}, \eqref{eq:weighted-linfty} for the Banach space 
$l^C_\infty(H_X)$). 
\begin{lemma} 
Assume that the genus-$g$ descendant potential $\cF^g$ 
is NF-convergent in the sense of Definition 
\ref{def:descendant-converge}.  
Then the specialization $Q_1= \cdots = Q_r=1$ of the 
right-hand side of \eqref{eq:divisorequation-cFg} depends only on 
the point $[e^{-\delta/z} \bq(z)]_+ \in H_X\otimes \C[\![z]\!]$. 
\end{lemma} 
\begin{proof} 
Suppose $[e^{-\delta/z} \bq(z)]_+ = [e^{-\delta'/z} \bq'(z)]_+$. 
We need to show that the specialization $Q_1= \cdots = Q_r=1$ of the 
right-hand side of \eqref{eq:divisorequation-cFg} 
has the same value at $(\bq,\delta)$ and $(\bq',\delta')$. 
This follows by applying \eqref{eq:divisorequation-cFg} itself to 
the relation $\bq'(z) = [e^{(\delta'-\delta)/z} \bq(z)]_+$. 
\end{proof} 

The lemma allows us to define a holomorphic function 
$\cF^g_{X, \rm an}$ as follows.

\begin{definition-proposition} 
  \label{def-pro:specialization}
  Assume that the genus-$g$ descendant potential $\cF^g$ is 
 NF-convergent in the sense of Definition
  \ref{def:descendant-converge}.  
Recall the definition of the Banach space $l_{\infty}^C(H_X)$ 
  in Remark~\ref{rem:holomorphy}, and set:
  \[
  B_\epsilon\left(l_{\infty}^C(H_X)\right) = \left \{ \bt(z) \in
    l_{\infty}^C(H_X) : \|\bt\|_{\infty, C} <\epsilon \right\}
  \]
  Then there exists a holomorphic function:
  \begin{equation}
    \label{eq:analyticdescendantpot} 
    \cF^g_{X, \rm an} \colon 
    \bigcup_{\substack{\delta\in H^2(X;\C),\\
        \Re(\delta_i) < \log \epsilon }} 
    \left [ 
      e^{-\delta/z} \Big( {-\phi_0} z +
        B_\epsilon \left (l_{\infty}^C(H_X)\right) \Big) 
    \right ]_+
    \to \C 
  \end{equation} 
  such that
\begin{align} 
    \label{eq:def-cFg-an} 
\begin{split} 
    \cF^g_{X, \rm an}([e^{-\delta/z} \bq(z)]_+) 
    & = \cF^g_X(\bq, e^{\delta_1},\dots, e^{\delta_r}) \\
 & +  
\frac{\delta_{g,0}}{2} \Omega(e^{-\delta/z} \bq(z), [e^{-\delta/z} \bq(z)]_+) 
- 
\frac{\delta_{g,1}}{24} \int_X \delta \cup c_{D-1}(X) 
\end{split}  
\end{align} 
  We refer to $\cF^g_{X,\rm an}$ as \emph{the specialization 
    of $\cF^g_X$ to $Q_1=\cdots =Q_r =1$}. 
\end{definition-proposition}

\subsection{Fundamental Solution} 

Recall the definition of the Dubrovin connection $\nabla$ in
\S\ref{sec:Dubrovinconnection}. Consider the \emph{fundamental
  solution} $L\in \End(H_X)\otimes \Lambda[\![t]\!][\![z^{-1}]\!]$
defined by:
\begin{equation}
  \label{eq:fundamentalsolution}
  L(t,z) v = v + 
  \sum_{d \in \NE(X)} \sum_{n=0}^{\infty} \sum_{\epsilon = 0}^{N} 
  \frac{Q^d}{n!} 
  \corr{\frac{v}{z-\psi},t,\ldots,t,\phi^\epsilon}^X_{0,n+2,d}
  \phi_\epsilon
\end{equation}
where $v \in H_X$. The expression $v/(z-\psi)$ in the correlator 
should be expanded in the series $\sum_{n=0}^\infty 
{v \psi^n}{z^{-n-1}}$.  
The fundamental solution satisfies:
\begin{align*}
  & \nabla_{\parfrac{}{t^i}}\big(L(t,z) z^{-\mu} z^{-\rho} v \big) = 0 \\
  & \nabla_{z \parfrac{}{z}}\big(L(t,z) z^{-\mu} z^{-\rho} v \big) = 0 
\end{align*}
for all $v \in H_X$, where $\rho = c_1(X)$ and the endomorphisms
$z^{-\mu}$ and $z^{-\rho}$ of $H_X$ are defined by $z^{-\mu} =
\exp({-\mu} \log z)$ and $z^{-\rho} = \exp({-\rho} \log z)$. 
The fundamental solution also satisfies:
\[
\left(L(t,-z) v, L(t,z) w\right) = (v,w) 
\]
for $v,w \in H_X$, where $(\cdot,\cdot)$ denotes the Poincar\'{e}
pairing of $H_X$, and so the inverse fundamental solution $M(t,z)=
L(t,z)^{-1}$ coincides with the adjoint of $L(t,-z)$:
\begin{equation} 
  \label{eq:inversefundsol} 
  M(t,z) v:= v + 
  \sum_{d \in \NE(X)} \sum_{n=0}^{\infty} \sum_{\epsilon = 0}^{N} 
  \frac{Q^d}{n!} 
  \corr{\frac{\phi^\epsilon}{-z-\psi},t,\ldots,t,v}^X_{0,n+2,d}
  \phi_\epsilon. 
\end{equation} 
The Divisor Equation for descendant invariants
\cite{AGV}*{Theorem~8.3.1} implies that:
\begin{equation} 
\label{eq:L-divisoreq} 
L(t,z) v = e^{\delta/z} v +  
\sum_{d \in \NE(X)} \sum_{n=0}^{\infty} \sum_{\epsilon = 0}^{N} 
\frac{e^{d\cdot \delta}Q^d}{n!} 
\corr{\frac{e^{\delta /z} v }{z-\psi}, t' ,\ldots, t', 
\phi^\epsilon}^X_{0,n+2,d}
\phi_\epsilon, 
\end{equation} 
where $t =\delta + t'$, $\delta \in H^2(X)$, $t' \in  
\bigoplus_{p\neq 1} H^{2p}(X)$.

If the genus-zero Gromov--Witten potential $F^0_X$ converges in the
sense of \S\ref{sec:convergence} then the fundamental solution with
$Q_1=\cdots = Q_r=1$ depends analytically on both $t$ and $z$, where
$t$ lies in a neighbourhood \eqref{eq:LRLnbhd} of the large-radius
limit point and $z$ is any point of $\Cstar$.

\subsection{Ancestor--Descendant Relation} 
In this section we distinguish the variables for descendant potentials
and ancestor potentials.  Let $\bx=(x_0,x_1,x_2,\dots)$ be a sequence
of variables in $H_X$ with $x_i = \sum_{\alpha=0}^N x_i^\alpha
\phi_\alpha \in H_X$.  Let $\bq =(q_0,q_1,q_2,\dots)$ be a sequence of
variables in $H_X$ with $q_i = \sum_{\alpha=0}^N q_i^\alpha
\phi_\alpha \in H_X$ as before.  We consider the change of variables:
\begin{equation} 
\label{eq:anc-dec-variables}
\bq(z) = [M(t,z) \bx(z)]_+  
\end{equation} 
between $\bx$ and $\bq$ for some $t\in H_X$.  Here $[\cdots]_+$
denotes the truncation of a $z$-series, $\bx(z) = \sum_{i=0}^\infty
x_i z^i$, $\bq(z) = \sum_{i=0}^\infty q_iz^i$ and $M(t,z)$ is the
inverse fundamental solution given in \eqref{eq:inversefundsol}.  We
relate the variables $\bq$, $\bx$ with the variables $\bt$, $\by$ by
the dilaton shift (cf.~\S\ref{sec:dilaton}):
\begin{align} 
\label{eq:Dilaton:anc-dec}
\bq(z) = \bt(z) -  \phi_0 z
&&
\bx(z) = \by(z) - \phi_0 z 
\end{align}
As in \S \ref{sec:GW-potentials}, we use $\bt$ as arguments for the
descendant potential $\cF^g_X$ and $\cZ_X$ and use $t, \by$ as
arguments for the ancestor potential $\bar\cF^g_X$ and $\cA_X$,
i.e.~$\cZ_X$ is a formal power series in $\bt$ and $\cA_X$ is a formal
power series in $t$ and $\by$.

\begin{theorem}[{Kontsevich--Manin 
\cite[Theorem 2.1]{Kontsevich--Manin:relations}, 
Givental \cite[\S 5]{Givental:quantization}, 
Coates--Givental \cite[Appendix 2]{Coates-Givental}}] 
\label{thm:ancestor-descendant}
Let $\cF^g$ and $\bar\cF^g_t$ denote the genus-$g$ descendant and
ancestor potentials of a smooth projective variety $X$.  We have:
\[
\cF^g(\bq) = 
\begin{cases}
   F^1(t) + \bar{\cF}^1_t(\bx) & \text{if $g=1$} \\
   \bar{\cF}^g_t(\bx) & \text{if $g \geq 2$}
\end{cases}
\]
under the change of variables given in \eqref{eq:anc-dec-variables}. 
Here $F^1$ is the non-descendant genus-$1$ potential. 
\end{theorem} 

\begin{remark} 
  In terms of the dilaton-shifted co-ordinates
  \eqref{eq:Dilaton:anc-dec}, the change of variables
  \eqref{eq:anc-dec-variables} can be written as:
  \begin{align}
    \label{eq:expansion-changeofvar}  
    \begin{split}
      t_0 & = t + y_0 + M_1(t) y_1 + M_2(t) y_2 + \cdots \\
      t_1 & = y_1 + M_1(t) y_2 + M_2(t) y_3+ \cdots \\
      t_2 & = y_2 + M_1(t) y_3 + M_2(t) y_4 + \cdots \\ 
      & \ \, \vdots 
    \end{split} 
  \end{align} 
  Here we write $M(t,z) = \Id + \sum_{n=1}^\infty M_n(t)z^{-n}$ and
  use $[M(t,z) (-\phi_0 z)]_+ = -\phi_0 z + t$.  This defines an
  isomorphism:
  \[
  \Lambda[\![\bt]\!][\![t]\!] \overset{\cong}{\longrightarrow}
  \Lambda[\![\by]\!][\![t]\!]
  \]
  because $M_n(0) \equiv 0$ modulo $(Q_1,\dots,Q_r)$ for $n\ge 1$.
\end{remark} 

\subsection{Analytic Preliminaries} 
\label{sec:Giventalnuclear}
Consider the family of Hilbert norms $\|\cdot\|_n$, $n=0,1,2,\dots$ on
$\C[\![z,z^{-1}]\!]$:
\begin{align*}
  \|\ba(z) \|_n = \left ( 
    \sum_{j\in \Z} \frac{|a_j|^2}{|\Gamma(\frac{1}{2}+j)|^2} 
    e^{2nj} \right )^{\frac{1}{2}}  
&& \text{where $\ba(z) = \sum_{j\in \Z} a_j z^j$}
\end{align*}
and set:
\[
\C\{\!\{z,z^{-1}\}\!\} = \left\{
\ba(z) \in \C[\![z,z^{-1}]\!] : \|\ba(z)\|_n < \infty \text{ for all } n\gg 0
\right\}. 
\]
We write :
\begin{align*} 
  \C\{\!\{z\}\!\} & = \C[\![z]\!] \cap \C\{\!\{z,z^{-1}\}\!\} = 
  \{ \ba(z) \in \C[\![z]\!] : \text{$\|\ba(z)\|_n < \infty$ for all $n
    \geq 0$}\}, \\ 
  \C\{\!\{z^{-1}\}\!\} &= \C[\![z^{-1}]\!] \cap \C\{\!\{z,z^{-1}\}\!\} 
  = \{ \ba(z) \in \C[\![z^{-1}]\!] 
  : \text{$\|\ba(z)\|_n < \infty$ for some $n \geq 0$}\}.  
\end{align*} 
Note that the norms are increasing $\|\cdot\|_0 \le\|\cdot\|_1\le
\|\cdot\|_2 \le \cdots $ on $\C\{\!\{z\}\!\}$ and $\C\{\!\{z\}\!\}$ is
a nuclear Fr\'{e}chet space whose topology is defined by these norms.
The norms are decreasing $\|\cdot \|_0\ge \|\cdot \|_1 \ge \|\cdot\|_2
\ge \cdots$ on $\C\{\!\{z^{-1}\}\!\}$ and $\C\{\!\{z^{-1}\}\!\}$ is an
inductive limit of Hilbert spaces; $\C\{\!\{z^{-1}\}\!\}$ with the
inductive limit topology is the strong dual of $\C\{\!\{z\}\!\}$ and
is a nuclear (DF) space.  The following Lemma shows that
$\C\{\!\{z,z^{-1}\}\!\}$ is a topological ring.
\begin{lemma} 
  \label{lem:product-estimate}
  For $\ba(z)$, $\bb(z) \in \C\{\!\{z,z^{-1}\}\!\}$, the product
  $\ba(z) \bb(z)$ converges. Moreover, we have:
  \begin{align*} 
    \|[\ba(z) \bb(z)]_+\|_{n-1} & \le 
    5
    (\|\ba(z)\|_{n+2} + \|\ba(z)\|_{n-2})
    (\|\bb(z)\|_{n+2} + \|\bb(z)\|_{n-2}) 
    \\
    \|[\ba(z) \bb(z)]_- \|_{n+1} & \le
    5
    (\|\ba(z)\|_{n+2} + \|\ba(z)\|_{n-2})
    (\|\bb(z)\|_{n+2} + \|\bb(z)\|_{n-2}) 
  \end{align*} 
  where $[\cdots]_+$ and $[\cdots]_-$ denote respectively the
  non-negative and strictly negative truncation of a power series in
  $z$ and $z^{-1}$.  In particular, if $\ba(z)\in
  \C\{\!\{z^{-1}\}\!\}$ and $\bb(z) \in \C\{\!\{z\}\!\}$, then:
  \begin{align}
    \label{eq:20}  
    \begin{split} 
      \|[\ba(z) \bb(z)]_+\|_{n-1} & \le 
      20 \|\ba(z)\|_{n-2} \|\bb(z)\|_{n+2} 
      \\
      \|[\ba(z) \bb(z)]_- \|_{n+1} & \le
      20 \|\ba(z)\|_{n-2} \|\bb(z)\|_{n+2}.  
    \end{split} 
  \end{align} 
\end{lemma} 
\begin{proof} 
  Observe first that:
  \[
  \left |
    \frac{\Gamma(j+\frac{1}{2})\Gamma(k+\frac{1}{2})}
    {\Gamma(j+k+\frac{1}{2})} \right|
  \le \pi e^{2|j|+2|k|} 
  \]
  for all $j,k\in \Z$.  Setting $\bc(z)= \sum_{l\in \Z} c_l z^l =
  \ba(z) \bb(z)$, we have:
  \begin{align*} 
    \left | \frac{c_l}{\Gamma(l+\frac{1}{2})} \right| 
    & \le  \sum_{j+k=l} \frac{|a_j b_k|}{|\Gamma(l + \frac{1}{2})|} 
    \le  \sum_{j+k=l} \frac{|a_j|}{|\Gamma(j+ \frac{1}{2})|} 
    \frac{|b_k|}{|\Gamma(k+\frac{1}{2})|} 
    \left| 
      \frac{\Gamma(j+\frac{1}{2})\Gamma(k+\frac{1}{2})}
      {\Gamma(j+k+\frac{1}{2})} \right |  \\
    & \le \pi e^{-nl} 
    \sum_{j+k=l} \frac{|a_j|}{|\Gamma(j+ \frac{1}{2})|} 
    e^{nj + 2|j|} 
    \frac{|b_k|}{|\Gamma(k+\frac{1}{2})|} e^{nk + 2 |k|} \\ 
    & \le \pi e^{-nl} (\|\ba(z)\|_{n+2} + \|\ba(z)\|_{n-2})
    (\|\bb(z)\|_{n+2} + \|\bb(z)\|_{n-2}) 
  \end{align*}   
  where we used the Cauchy--Schwarz inequality in the last step. The
  conclusion follows.
\end{proof} 

\begin{remark} 
  Let $\tau$ be the co-ordinate Laplace-dual to $z^{-1}$ and let $(j_*
  \cO_{\C_\tau})_\infty$ denote the space of germs of holomorphic
  functions $f$ defined on a small punctured neighbourhood:
  \[
  \{\tau \in \C : R_f<|\tau|<\infty\}
  \]
  of $\tau=\infty$. Here $j\colon \C_\tau \hookrightarrow
  \Proj^1_\tau$ is the natural inclusion.  A calculation similar to
  that in the proof of Lemma \ref{lem:product-estimate} shows that the
  ring $\C\{\!\{z,z^{-1}\}\!\}$ acts on $(j_* \cO_{\C_\tau})_\infty$
  as microdifferential operators:
  \begin{align*}
    f(\tau) \mapsto \tau^{1/2}\ba(\partial_\tau^{-1}) \tau^{-1/2}f(\tau), 
    && \text{ for $\ba(z) \in \C\{\!\{z,z^{-1}\}\!\}$} 
  \end{align*}
  Here $\tau^{-1/2}$ was put to make the action well-defined.  The
  positive part $\C\{\!\{z\}\!\}$ preserves the space of entire
  functions $\cO(\C_\tau) \subset (j_* \cO_{\C_\tau})_\infty$ and the
  negative part $\C\{\!\{z^{-1}\}\!\}$ preserves the space of germs of
  holomorphic functions $(\cO_{\Proj^1_\tau})_\infty \subset
  (j_*\cO_{\C_\tau})_\infty$ at $\tau=\infty$.
\end{remark} 

\begin{definition}[cf.~\cite{Givental:symplectic}]
  We now define a nuclear version of \emph{Givental's symplectic
    space}.  This is a vector space:
  \[
  \cH = H_X \otimes \C\{\!\{z,z^{-1}\}\!\}
  \]
  equipped with Givental's symplectic form:
  \begin{align*}
    \Omega\colon 
    \cH  \times \cH & \longrightarrow  \C \\
    (\bbf(z), \bg(z)) & \longmapsto 
    \Res_{z=0} \langle \bbf(-z), \bg(z)\rangle_{H_X} dz 
  \end{align*}
  It has the standard polarization $\cH = \cH_+ \oplus \cH_-$,
  where\footnote{$\cH_+$ here coincides with the previous formula
    \eqref{eq:H+}.}:
  \begin{align*}
    \cH_+ := H_X \otimes \C\{\!\{z\}\!\} &&
    \cH_- := H_X \otimes z^{-1} \C\{\!\{z^{-1}\}\!\}.  
  \end{align*}
  The symplectic form $\Omega$ identifies $\cH_-$ with the strong dual
  of $\cH_+$, and identifies $\cH_+$ with the strong dual of $\cH_-$.
  The spaces $\cH$, $\cH_\pm$ are fully nuclear; $\cH_+$ is
  Fr\'{e}chet and $\cH_-$ is (DF).
\end{definition} 

\begin{lemma} 
  \label{lem:ML-estimate} 
  Assume that the genus-zero non-descendant potential $F^0_X$
  converges as in \S\ref{sec:convergence}.  Let $L(t,z)$ be the
  fundamental solution \eqref{eq:fundamentalsolution}, and let
  $M(t,z)$ be the inverse fundamental solution
  \eqref{eq:inversefundsol}.  Then there exist $\epsilon>0$, $n\ge 0$,
  and $R = R(\alpha_1,\dots,\alpha_l,j_1,\dots,j_m)>0$ such that for
  $|t^\alpha|<\epsilon$, $|Q_i|<\epsilon$ the matrix entries of:
  \begin{align*}
    \parfrac{^{l+m}L(t,z)}{t^{\alpha_1} \cdots \partial t^{\alpha_l} 
      \partial Q_{j_1} \cdots \partial Q_{j_m}}
    && \text{and} && 
    \parfrac{^{l+m}M(t,z)}{t^{\alpha_1} \cdots \partial t^{\alpha_l} 
      \partial Q_{j_1} \cdots \partial Q_{j_m}} 
  \end{align*}
  lie in the bounded subset:
  \[
  \{ \ba(z) \in \C\{\!\{z^{-1}\}\!\} : \|\ba(z)\|_n < R\}.  
  \] 
\end{lemma} 
\begin{proof} 
  Writing $q^\bm = (t^0)^{m_0} (Q_1 e^{t^1})^{m_1} \cdots (Q_r
  e^{t^r})^{m_r} (t^{r+1})^{m_{r+1}} \cdots (t^N)^{m_N}$, we can
  expand the inverse fundamental solution as:
  \[
  M(t,z) = e^{-t/z} \sum_{i = 0}^\infty \sum_{\bm} M_{\bm, i} 
  q^{\bm} z^{-i} 
  \]
  with $M_{\bm,i} \in \End(H_X)$.  It was shown\footnote{Note that
    what is denoted by $M(t,z)$ here is denoted by $L(t,-z)$ in
    \cite{Iritani:convergence}.} in
  \cite{Iritani:convergence}*{Lemma~4.1} that $M_{\bm,i}$ satisfies:
  \[
  \|M_{\bm,i} \| \le A C^{|\bm| + i}\frac{1}{i!}  
  \]
  for some $A, C>0$.  The conclusion about the partial derivatives of
  $M(t,z)$ follows from this.  The same argument as
  \cite{Iritani:convergence}*{Lemma 4.1} shows the same estimates for
  the coefficients of $L(t,z)$.  This implies the conclusion about the
  partial derivatives of $L(t,z)$.
\end{proof} 

Let $\C\{\!\{z\}\!\}_\infty^n$ be the local Banach space of the
Fr\'{e}chet space $\C\{\!\{z\}\!\}$ associated to the norm
$\|\ba(z)\|_{\infty, n} = \sup_j (|a_j| e^{nj}/j!)$, i.e.:
\[
\C\{\!\{z\}\!\}_\infty^n= \left \{ \ba(z) \in \C[\![z]\!] :
  \sup_j \left(\frac{|a_j| e^{nj}}{j!}\right) < \infty \right\} \supset \C\{\!\{z\}\!\}
\] 
Then $l_\infty^C(H_X) \cong H_X \otimes \C\{\!\{z\}\!\}_\infty^n$ 
for $C=e^n$: see \eqref{eq:l_infinityC}.  

\begin{lemma} 
\label{lem:Frechet-Banach} 
{\rm (1)} 
Let $\cF(\ba)$ be a holomorphic function defined on a
    neighbourhood of the origin of the Fr\'{e}chet space
    $\C\{\!\{z\}\!\}$.  Consider the monomial Taylor expansion:
    \begin{equation}
      \label{eq:monomialTaylorexpansion}
      \sum_{\bm = (m_0,m_1,m_2,\dots) } 
      \frac{1}{m_0! m_1! m_2! \cdots}
      \parfrac{^{|\bm|} \cF}
      {a_0^{m_0} \partial a_1^{m_1} \partial a_2^{m_2}\cdots}
      (0) a_0^{m_0} a_1^{m_1} a_2^{m_2}\cdots 
    \end{equation}
    where $\bm = (m_0,m_1,m_2,\dots) $ is a sequence of non-negative
    integers such that $m_i = 0$ for $i \gg 0$, and $|\bm| = \sum_i
    m_i$.  There exist $n\ge 0$ and $\epsilon>0$ such that the
    monomial Taylor expansion \eqref{eq:monomialTaylorexpansion}
    converges absolutely and uniformly on the $\epsilon$-ball:
    \[
    B_\epsilon (\C\{\!\{z\}\!\}_{\infty}^n) = \{\ba \in
    \C\{\!\{z\}\!\}_{\infty}^n : \|\ba\|_{\infty,n} < \epsilon\}
    \] 
    in the Banach space $\C\{\!\{z\}\!\}_{\infty}^n$ and coincides with
    $\cF(\ba)$ for $\ba \in \C\{\!\{z\}\!\}\cap
    B_\epsilon(\C\{\!\{z\}\!\}_{\infty}^n)$. 

{\rm (2)} If $\cF(\ba)$ is uniformly continuous with respect to the norm
    $\|\cdot\|_{\infty, p}$ and holomorphic on the ball $\{\ba(z) \in
    \C\{\!\{z\}\!\}:\|\ba (z) \|_{\infty, p} < \rho\}$, then the above
    monomial Taylor expansion \eqref{eq:monomialTaylorexpansion}
    converges absolutely and uniformly on the ball $B_{\rho/3}
    (\C\{\!\{z\}\!\}_{\infty}^{p+2})$.
\end{lemma} 
\begin{proof} 
  Let us write the monomial Taylor expansion
  \eqref{eq:monomialTaylorexpansion} as:
  \[
  \sum_{\bm} \frac{1}{\bm !} \cF^{(\bm)}(0) \ba^{\bm}.
  \]
  There exist $\eta>0$ and $n\ge 0$ such that $\cF$ is holomorphic
  on:
  \[
  \{\ba\in \C\{\!\{z\}\!\} : \|\ba\|_{\infty,n}<2\eta\}
  \]
  Decreasing $\eta$ and increasing $n$ if necessary, we can assume
  that $\cF$ is bounded on: 
  \[
  \{\ba \in\C\{\!\{z\}\!\} : \|\ba\|_{\infty,n} \le \eta\}
  \]
  since $\cF$ is continuous.  Set $R_j = j! e^{-n j} \eta $.  By
  the Cauchy integral formula applied to $\cF(a_0,\dots,a_l,
  0,0,\dots)$, we obtain for $\bm = (m_0,\dots,m_l,0,0,\dots)$:
  \begin{align*} 
    \left | \frac{1}{\bm!} \cF^{(\bm)}(0) \right| 
    & = 
    \left | 
      \frac{1}{(2\pi\iu)^{l+1}}
      \int_{\text{$|a_j|= R_j$, $0 \leq j \leq l$}} 
      \frac{\cF(a_0,\dots,a_l,0,\dots)}{a_0^{m_0+1} 
        \cdots a_l^{m_l + 1}} da_0 \cdots da_l  \right| \\ 
    & \le  \frac{M e^{n \sum_{j} j m_j}}{\eta^{|\bm|} \prod_i (i!)^{m_i}} 
  \end{align*} 
  where $M$ is the supremum of $|\cF(\ba)|$ over $\{
  \|\ba\|_{\infty,n} \le \eta\}$.  Set $\eta = e \epsilon$.  Then if
  $\|\ba\|_{\infty,n+1} < \epsilon$, we have:
  \[
  \left | \frac{1}{\bm!} \cF^{(\bm)}(0) \ba^{\bm} \right| 
  \le M e^{-\sum_i (i+1) m_i} 
  \]
  The right-hand side is absolutely convergent because 
  \[
  \sum_{\bm} e^{-\sum_i (i+1) m_i}  = \prod_{i=0}^\infty 
  \frac{1}{1- e^{-i-1}} < \infty. 
  \]
  Hence the monomial Taylor expansion
  \eqref{eq:monomialTaylorexpansion} converges absolutely and
  uniformly in the ball $B_{\epsilon}(\C\{\!\{z\}\!\}_\infty^{n+1})$.
  The Taylor series and $\cF(\ba)$ match for $\ba \in
  \C\{\!\{z\}\!\}$ with $\|\ba\|_{\infty,n+1} < \epsilon$
  as both are continuous and they match on the dense subset $\{\ba \in
  \C[z] : \|\ba\|_{\infty,n+1}< \epsilon\}$. 
 This proves Part~(1). 
 
 Part~(2) can be proved by a small modification of the above argument.
 Because $\cF$ is uniformly continuous with respect to
 $\|\cdot\|_{\infty, p}$, it extends uniquely to a $\|\cdot\|_{\infty,
   p}$-continuous function on the ball
 \[
 B= \left\{\ba(z) \in \C\{\!\{z\}\!\}_{\infty}^p 
   : \|\ba(z)\|_{\infty,p} < \rho \right\} 
 \]
 In view of the above it suffices to show, under the hypotheses of
 Part~(2), that $\cF$ is bounded on:
 \[
 B' = \left\{\ba(z) \in \C\{\!\{z\}\!\} : 
   \|\ba(z) \|_{\infty, p+1} \le e \rho/3 \right\} 
 \]
 Suppose on the contrary that $\cF$ is not bounded on $B'$. Then there
 exists a sequence $(\ba_l)_{l=1}^\infty$ in $B'$ such that
 $\lim_{l\to \infty} |\cF(\ba_l)| =\infty$.  Because
 $(\ba_l)_{l=1}^\infty$ is bounded in the norm
 $\|\cdot\|_{\infty,p+1}$, one can find a subsequence
 $(\ba_{l_n})_{n=1}^\infty$ which converges to an element in $B$ in
 the norm $\|\cdot\|_{\infty,p}$.  But $\cF$ extends to a continuous
 function on $B$, so this is a contradiction.
\end{proof} 

\begin{remark} 
  In infinite dimension there are two different Taylor expansions:
  monomial expansion as above and the expansion $\sum_{m=0}^\infty
  P_m(\ba,\dots,\ba)$ by $m$-linear forms $P_m = \frac{1}{m!} D^m_0
  \cF$.  For a holomorphic function on $\C\{\!\{z\}\!\}_{\infty}^n$,
  the monomial expansion does not necessarily converges whereas
  $\sum_{m=0}^\infty P_m(\ba,\dots,\ba)$ always does: see \cite{DMP}
  and references therein.  On the other hand, Boland--Dineen
  \cite{Boland-Dineen} showed that monomials form an absolute basis of
  the space of holomorphic functions on the open set $\{\ba \in
  \C\{\!\{z\}\!\} : \|\ba\|_{\infty,n} < \epsilon\}$ in
  $\C\{\!\{z\}\!\}$ with respect to a certain topology $\tau_\omega$.
\end{remark} 

\subsection{NF-Convergence of the Genus-Zero Descendant Potential}
\label{subsec:NF-conv-genuszero}

In this section we prove Theorem~\ref{thm:convergence-genuszero}: that
Genus-Zero Convergence \eqref{eq:genuszeroconvergence}, which is a
convergence assumption on the non-descendant genus-zero potential
$F^0_X$, implies the NF-convergence of the descendant genus-zero
potential $\cF^0_X$.  The main ingredients are the Nash--Moser inverse
function theorem and the reconstruction theorem of Dubrovin and
Dijkgraaf--Witten
\citelist{\cite{Dubrovin:2DTFT}\cite{Dijkgraaf-Witten:meanfield}},
which determines descendant genus-zero invariants from primary
genus-zero invariants.

We introduce a sequence of variables $\bp=(p_0,p_1,p_2,\dots)$ in
$H_X$ with $p_i = \sum_{\alpha=0}^N p_{i,\alpha} \phi^\alpha$, and a
generating function
\[
\bp(z) = \sum_{i=0}^\infty \sum_{\alpha=0}^N 
p_{i,\alpha} \frac{\phi^\alpha}{(-z)^{i+1}}  
\]
taking values in $z^{-1}H_X[\![z^{-1}]\!]$.  Let $M(t,z)$ denote the
inverse fundamental solution \eqref{eq:inversefundsol}.  Consider the
ancestor variable $\bx =(0,x_1,x_2,\dots)$ with $x_0=0$ and set:
\begin{equation} 
\label{eq:bq+bp}
\bq(z) + \bp(z) = M(t,z) \bx(z) 
\end{equation} 
(cf.~equation~\ref{eq:anc-dec-variables}) where $\bq(z)$ is the
non-negative part and $\bp(z)$ is the strictly negative part.  Recall
that $\bx$ and $\by$ are related by the dilaton shift
\eqref{eq:Dilaton:anc-dec}.  Because the $0$th ancestor variable $x_0
= y_0$ is now set to equal zero, the map:
\begin{equation} 
\label{eq:tbx-bq} 
(t,\bx(z) ) \mapsto \bq(z) = [M(t,z) \bx(z)]_+ 
\end{equation} 
defines an isomorphism between the formal neighbourhoods 
of $\by_{\ge 1}=t=0$ and $\bt=0$: 
\[
\Lambda[\![\bt]\!] \overset{\cong}{\longrightarrow} 
\Lambda[\![\by_{\ge 1}]\!][\![t]\!]. 
\]
(This is clear from equation
\ref{eq:expansion-changeofvar}.)\phantom{.}  Equation
\eqref{eq:bq+bp} determines $\bp$, $\bq$ as formal power series in
$\by_{\ge 1}$ and $t$.  Via the isomorphism above we can regard
$p_{i,\alpha}$ as a formal power series in $\bt$, i.e.~as an element
of $\Lambda[\![\bt]\!]$.

\begin{theorem}[Dubrovin \cite{Dubrovin:2DTFT}, Dijkgraaf--Witten
  \cite{Dijkgraaf-Witten:meanfield}]
  \label{thm:reconstruction} 
  The descendant Gromov--Witten potential $\cF_X^0$ is given by
  \[
  \cF^0_X = \frac{1}{2}\Omega(\bp(z),\bq(z)) 
  = \frac{1}{2} \sum_{i=0}^\infty \sum_{\alpha=0}^N 
  p_{i,\alpha} q_i^\alpha. 
  \]
  Here $\Omega$ is Givental's symplectic form defined in
  \eqref{eq:symplecticform}. 
\end{theorem} 

\begin{proof} 
  Note that the right-hand side converges in the adic topology of
  $\Lambda[\![\bt]\!]$, because $v(q_i^\alpha) = i+1$ for $i\ge 2$.
  We use a reformulation by Givental
  \cite{Givental:quantization}*{\S5}, proven in Appendix~2 of
  \cite{Coates-Givental}.  The inverse of the co-ordinate change
  \eqref{eq:tbx-bq} is given by the fundamental solution $L(t,z)$ in
  \eqref{eq:fundamentalsolution} as:
  \begin{align*}
    &0 = [L(t,z) \bq(z)]_0 &
    &\bx(z) = [L(t,z) \bq(z)]_{\ge 1} 
  \end{align*}
  where $[\cdots]_0$ means the coefficient of $z^0$ and $[\cdots]_{\ge
    1}$ means the strictly positive truncation of a power series in
  $z$.  The first equation implicitly determines $t$ as a function
  $t(\bq)$ of $\bq$.  We have:
  \begin{align*} 
    \frac{1}{2} \Omega(\bp(z),\bq(z)) & = \frac{1}{2}
    \Omega \left( [M(t,z) \bx(z)]_- , [M(t,z) \bx(z)]_+\right) \\
    & = \frac{1}{2} \Omega\left(
      M(t,z) [L(t,z) \bq(z)]_+, \bq(z)\right)  \\
    & = \frac{1}{2} \Omega\left( 
      [L(t,z) \bq(z)]_+, L(t,z) \bq(z)\right)
  \end{align*} 
  with $t = t(\bq)$.  This coincides with $W_t(\bq,\bq)/2$ in
  \cite{Givental:quantization}*{Proposition~5.3}.
\end{proof} 

\subsection*{Proof of Theorem~\ref{thm:convergence-genuszero}}
We set:
\begin{align*}
  & \cH_{\ge 1}= \left\{ \bx(z) \in \cH_+ :x_0 = 0 \right \}
  \\
  &\Delta_\epsilon = \{ a\in \C : |a|<\epsilon\}
\end{align*}
By Lemma~\ref{lem:ML-estimate} and our convergence assumption for
$F_X^0$, there exist $n\ge 0$ and $\epsilon>0$ such that all the
matrix entries of $M(t,z)$ with $|t^\alpha|<\epsilon$,
$|Q_i|<\epsilon$ are bounded with respect to the norm $\|\cdot \|_n$.
Therefore if $\bx\in \cH_{\ge 1}$, $|t^\alpha|<\epsilon$, and
$|Q_i|<\epsilon$, $(\bp,\bq)$ defined by the equation \eqref{eq:bq+bp}
lies in $\cH_- \times \cH_+$ by Lemma~\ref{lem:product-estimate}, and
the sum $\sum_{i=0}^\infty \sum_{\alpha=0}^N p_{i,\alpha} q_i^\alpha$
converges.  Moreover the map:
\begin{align*}
  \cH_{\ge 1}\times \Delta_{\epsilon}^{N+1+r} & \to \cH_-\times \cH_+ \\
  (\bx,t,Q) & \mapsto (\bp,\bq)
\end{align*}
given by \eqref{eq:bq+bp} is continuous, because:
\begin{multline*} 
\left\|M(t,z;Q) \bx(z) - M(t',z;Q') \bx'(z)\right \|_n 
\le  A 
\left \| (t,Q)- (t',Q') \right\|  \|\bx(z) \|_{n+3} \\ 
+ B \|\bx(z) - \bx'(z)\|_{n+3} 
\end{multline*}
for $n\gg 0$ and some $A, B>0$.  This follows from the estimate
\eqref{eq:20} and the uniform estimate of the derivatives of
$M(t,z;Q)$ in Lemma~\ref{lem:ML-estimate}.  The map:
\[
(\bx,t,Q) \mapsto \frac{1}{2}\Omega\big(\bp(z),\bq(z)\big)
\]
is obviously G\^{a}teaux-differentiable, and therefore defines a
holomorphic function of $(\bx,t,Q) \in \cH_{\ge 1}\times
\Delta_{\epsilon}^{N+1+r}$ (see Remark~\ref{rem:holomorphy}).  This
gives the genus-zero descendant potential $\cF_X^0$ by
Theorem~\ref{thm:reconstruction}.

In view of Lemma~\ref{lem:Frechet-Banach}, it now suffices to show
that the map $(\bx, t, Q) \mapsto (\bq,Q)$ given by \eqref{eq:tbx-bq}
defines a local isomorphism between a neighbourhood of $(\bx, t, Q) =
(-\phi_0 z, 0, 0)$ in $\cH_{\ge 1} \times \C^{N+1+r}$ and a
neighbourhood of $(\bq,Q) = (-\phi_0 z, 0)$ in $\cH_+\times \C^r$.  We
apply the Nash--Moser inverse function theorem
\cite{Hamilton}*{Part~III,~Theorem 1.1.1}.  We need to show that there
exists a neighbourhood $U\subset \cH_{\ge 1}\times \C^{N+1+r}$ of
$(\bx,t,Q) = (- \phi_0 z, 0,0)$ such that:
\begin{itemize} 
\item the map $(\bx,t,Q) \mapsto (\bq,Q)$ is smooth tame  
  \cite{Hamilton}*{Part~II,~\S2.1} on $U$; 
\item the linearized operator $D_{\bx,t,Q}(\bq,Q)$ is invertible 
  at every $(\bx,t,Q)\in U$; 
\item the inverse $(D_{\bx,t,Q}(\bq,Q))^{-1}$ is continuous and 
  tame as a map $U\times (\cH_+\times \C^r) 
  \to \cH_{\ge 1}\times \C^{N+1+r}$ 
  (see \cite{Hamilton}*{Part~II,~Theorem~3.1.1}).  
\end{itemize} 
The proof of smoothness is similar to the proof of continuity
above and is omitted.  Because $\bq$ is linear in $\bx$, the
smooth-tameness of $(\bx,t,Q)\mapsto \bq$ follows from the inequality:
\begin{equation}
  \label{eq:inequality}
  \left \| \left[\partial_{v_1}\cdots \partial_{v_l} 
      M(t,z) \bx(z)\right ]_+\right \|_n
  \le A  \|\bx\|_{n+3}   
\end{equation}
for all $(t,Q)\in \Delta_\epsilon^{N+1+r}$, $n\gg 0$, and some $A>0$
($A$ can depend on $v_1,\dots,v_l$).  Here $\partial_{v_j}$ denotes
the partial derivative along $\Delta_\epsilon^{N+1+r}$.  The
inequality \eqref{eq:inequality} follows directly from the estimate
\eqref{eq:20} and Lemma~\ref{lem:ML-estimate}.

The linearized operator is given by:
\[
(D_{\bx,t,Q} (\bq,Q)) (d \bx, d t, dQ) = 
\left( 
  \left[M(t,z) (-z^{-1} dt * \bx + d\bx )\right]_+ 
  + \sum_{i=1}^r dQ_i
  \left[\parfrac{M}{Q_i}(t,z) \bx\right]_+, 
  dQ \right) 
\]
where $*$ denotes the analytic quantum product depending on $(t,Q)$.
Equating this with $(d\bq, dQ)$, we get:
\[
-z^{-1} dt * \bx + d\bx  
= \left[L(t,z) \left( 
d\bq - \sum_{i=1}^r dQ_i
\left[\parfrac{M}{Q_i}(t,z) \bx\right]_+\right) \right]_+
\]
The right-hand side is continuous and tame as a map from $((\bx,t,Q),
(d\bq, dQ) ) \in (\cH_{\ge 1} \times \Delta_\epsilon^{N+1+r})\times
(\cH_+ \times \C^r)$ to $\cH_+$ for the same reason as before.
Equating the left-hand side with $\bv =(v_0,\bv_{\ge 1})$ yields:
\begin{align*}
  & v_0 = - dt* x_1 &&
  \bv_{\ge 1} = z^{-1} dt* \bx_{\ge 2} + d\bx. 
  \intertext{When $x_0$ is sufficiently close to $-\phi_0$, the first equation can
    be inverted and we obtain the inverse map $((\bx,t,Q), \bv) \mapsto
    (d\bx,dt)$ given by:}
  & dt = - (x_1*)^{-1} v_0 &&
  d\bx = \bv_{\ge 1} + z^{-1} ((x_1*)^{-1} v_0)* \bx_{\ge 2}. 
\end{align*}
This map is continuous and tame. 
Hence the linearized operator admits a continuous and 
tame inverse in a neighbourhood of $(-\phi_0 z, 0,0)$. 
The Nash--Moser theorem now applies. \qed

\subsection*{An Analytic Version of Theorem~\ref{thm:reconstruction}}
We saw in \S\ref{sec:specializationmakesense} that whenever $\cF^0_X$
converges, we can define the specialization $\cF^0_{X,\rm an}$ of
$\cF^0_X$ to $Q_1 = \cdots = Q_r = 1$.  We now show that the
Dubrovin--Dijkgraaf--Witten reconstruction theorem
(Theorem~\ref{thm:reconstruction}) holds for this $\cF^0_{X,\rm an}$.
Genus Zero Convergence \eqref{eq:genuszeroconvergence} implies that
$M(t,z)|_{Q_1= \cdots =Q_r = 1}$ depends analytically on $(t,z)$ where
$z \in \Cstar$ and $t$ lies in a neighbourhood \eqref{eq:LRLnbhd} of
the large-radius limit point.  Set $\cH_{\ge 1}= \{\bx(z)\in \cH_+ :
x_0 = 0\}$. We define $\bq$ and $\bp$ by a formula analogous to
\eqref{eq:bq+bp}:
\begin{equation} 
\label{eq:bq+bp:Q=1} 
\bq(z) + \bp(z) = M(t,z)\Big|_{Q_1=\cdots=Q_r=1} \bx(z) 
\end{equation} 
where $\bx(z) \in \cH_{\ge 1}$ and:
\[
\bq(z) = \left[M(t,z)\Big|_{Q_1= \cdots =Q_r = 1} \bx(z)\right]_+
\]
Arguing as in the proof of Theorem \ref{thm:convergence-genuszero}
shows that one gets a continuous mapping $(t,\bx) \mapsto (\bp,\bq)
\in \cH_- \times \cH_+$, and that the map:
\begin{align*}
  \cH_{\ge 1}\times H_X & \to \cH_+ \\
  (\bx, t) & \mapsto \bq(z)
\end{align*}
gives a local isomorphism between $\cH_{\ge 1}\times H_X$ and $\cH_+$,
for $t$ in a neighbourhood \eqref{eq:LRLnbhd} of the large-radius
limit point and $\bx$ in a neighbourhood of $-\phi_0 z$.  Therefore,
for any chosen point $t\in H_X\otimes \C$ in a neighbourhood
\eqref{eq:LRLnbhd} of the large-radius limit point, the sum
$\frac{1}{2} \sum_{i=0}^\infty \sum_{\alpha=0}^\infty p_{i,\alpha}
q_i^\alpha$ can be regarded as a holomorphic function on a
neighbourhood of $\bq (z) = [M(t,z) (-\phi_0 z)]_+ = t-\phi_0 z$ in
$\cH_+$.

\begin{theorem}[analytic version of Theorem \ref{thm:reconstruction}] 
  Assume that the genus-zero descendant Gromov--Witten potential
  $\cF^0_X$ is NF-convergent in the sense of Definition
  \ref{def:descendant-converge}.  Then its specialization
  $\cF^0_{X,\rm an}$ to $Q_1=\cdots =Q_r=1$ (see
  \eqref{eq:analyticdescendantpot}) is given by:
  \[
  \cF_{X,\rm an}^0 = \frac{1}{2} 
  \Omega(\bp,\bq)= \frac{1}{2} 
  \sum_{i=0}^\infty \sum_{\alpha=0}^N p_{i,\alpha} q_i^\alpha 
  \]
  The right-hand side here is, as discussed in the preceding
  paragraph, a holomorphic function defined on a neighbourhood of
  $\bq(z) = t - \phi_0 z$ in $\cH_+$, where $t$ is a point in the
  neighbourhood \eqref{eq:LRLnbhd} of the large-radius limit point.
\end{theorem} 

\begin{proof} 
  We write the right-hand side as:
  \[
  C^{(0)}(\bq) := \frac{1}{2} \Omega(\bp,\bq).
  \]
  The Divisor Equation shows that (cf.~equation~\ref{eq:L-divisoreq}):
  \[
  M(t-\delta,z)\Big|_{Q_1=e^{\delta_1},\dots, Q_r=e^{\delta_r}}  
  = e^{\delta/z} M(t,z) \Big|_{Q_1=\cdots=Q_r=1}    
  \]
  where $\delta = \sum_{\alpha=1}^r \delta_\alpha \phi_\alpha$.  Therefore by
  \eqref{eq:bq+bp:Q=1} for $\bx =(0,x_1,x_2,\dots)$:
  \[
  M(t-\delta,z)\bx(z) \Big|_{Q_1=e^{\delta_1},\dots, Q_r=e^{\delta_r}} 
  = e^{\delta/z} (\bp(z) + \bq(z))
  \]
  Assume now that $\Re(\delta_i)\ll 0$ and that $(t-\delta, \bx)$ is
  sufficiently close to $(0,-\phi_0 z)$.  Setting:
  \begin{align*} 
    \tbq(z) & = \left[M(t-\delta,z)\bx(z) \right]_+ 
    \Big|_{Q_1=e^{\delta_1},\dots, Q_r=e^{\delta_r}} 
    = \left[e^{\delta/z} \bq(z)\right]_+ \\
    \tbp(z) & = \left[
      M(t-\delta,z)\bx(z)\right]_- 
    \Big|_{Q_1=e^{\delta_1},\dots, Q_r=e^{\delta_r}} 
  \end{align*} 
  we have from the definition of $C^{(0)}$ and the the original
  reconstruction Theorem~\ref{thm:reconstruction} that:
  \begin{align*}
    C^{(0)}(\bq) & = \frac{1}{2} \Omega(\bp,\bq) 
= \frac{1}{2} \Omega(e^{\delta/z} \bp, e^{\delta/z} \bq)  
= \frac{1}{2} \Omega(\tbp - [e^{\delta/z} \bq]_-, \tbq + 
[e^{\delta/z} \bq]_- )  \\
& = \frac{1}{2} \Omega(\tbp,\tbq) - \frac{1}{2} 
\Omega([e^{\delta/z} \bq]_-, \tbq)  
 = \frac{1}{2} \Omega(\tbp,\tbq) 
 - \frac{1}{2} \Omega(e^{\delta/z} [e^{-\delta/z} \tbq]_+, \tbq) 
\\
&= \cF_X^0(\tbq, e^{\delta_1},\dots,e^{\delta_r}) 
+ \frac{1}{2} 
\Omega( e^{-\delta/z} \tbq, [e^{-\delta/z} \tbq]_+)
  \end{align*}
  Applying \eqref{eq:def-cFg-an}, we see that the right-hand side
  coincides with $\cF_{X,\rm an}^0(\bq)$.
\end{proof} 

\subsection{The Proof of Theorem~\ref{thm:descendant-convergence}} 
\label{subsec:proof-descendant-convergence} 

The genus-zero ancestor potential $\bar\cF^0_X$ contains as a subseries:
\[
\sum_{d \in \NE(X)} \sum_{n=0}^\infty \sum_{0\le \alpha,\beta,\gamma \le N} 
\frac{Q^d}{3! n!} 
\corr{y_0^\alpha \phi_\alpha, y_0^\beta \phi_\beta, 
y_0^\gamma \phi_\gamma; \overbrace{t,\dots,t}^{n} 
}_{0,3+n,d}^X 
\]
and our convergence assumption implies that the coefficient of
$y_0^\alpha y_0^\beta y_0^\gamma$ converges as a power series in $t$
and $Q$.  This shows that all third derivatives of the non-descendant
genus zero potential $F^0_X$ are convergent, and thus that $F^0_X$
itself is convergent.  Theorem~\ref{thm:convergence-genuszero} then
implies that the genus-zero descendant potential $\cF^0_X$ is
NF-convergent.

A similar argument shows that the genus-one non-descendant potential
$F^1_X(t)$ converges.  All derivatives of $F^1_X(t)$ appear as
subseries of $\bar\cF^1_t$ and hence are convergent. Thus $F^1_X(t)$
is also convergent.  Now the Ancestor--Descendant Relation (Theorem
\ref{thm:ancestor-descendant}) leads to the NF-convergence of
higher-genus descendant potentials.  Under our convergence
assumptions, the ancestor potential $\bar\cF^g_t(\bx)$ with $x_0$ set
to equal zero depends:
\begin{itemize}
\item analytically on $t^i$, $Q_i$ in a region \eqref{eq:e-ball}, for
  some $\epsilon>0$;
\item rationally on $x_1$; and
\item polynomially on $x_2,x_3,\dots$.  
\end{itemize}
In particular it is holomorphic in a small neighbourhood of
$(\bx,t,Q)=(-\phi_0 z ,0 ,0)$ in the Fr\'{e}chet space 
$\cH_{\ge  1}\times \C^{N+1+r}$;  
moreover, for every $n\ge 0$, it is uniformly continuous with respect to 
the norm $\|\cdot \|_n$ in a $\|\cdot \|_n$-neighbourhood 
of $(-\phi_0 z, 0, 0)$. 
On the other hand, in the proof (\S \ref{subsec:NF-conv-genuszero}) of 
Theorem \ref{thm:convergence-genuszero}, 
we used the Nash--Moser inverse function theorem 
to show that the co-ordinate change
$(\bx,t, Q)\mapsto (\bq,Q)$ defined by \eqref{eq:tbx-bq} is an
isomorphism between a neighbourhood of $(-\phi_0 z,0, 0)$ 
in $\cH_{\ge  1}\times \C^{N+1+r}$ 
and a neighbourhood of $(-\phi_0 z, 0)$ in $\cH_+ \times \C^r$.  
The Nash--Moser theorem \cite{Hamilton}*{Part~III,~Theorem 1.1.1} 
moreover 
asserts that the inverse map $(\bq,Q) \mapsto (\bx,t,Q)$ 
is smooth tame. Therefore there exist $m\ge n\ge 0$ 
such that the inverse map is defined on a $\|\cdot\|_m$-neighbourhood 
of $(-z \phi_0,0)$ and is 
$(\|\cdot\|_m, \|\cdot\|_n)$-Lipschitz continuous\footnote
{To show this, we apply \cite{Hamilton}*{Part~II,~Lemma 2.1.7} 
to the derivative of the inverse map.} 
there, i.e.\ 
there exist $C>0$ such that 
\[
\|(\bx', t',Q') - (\bx,t,Q) \|_n \le C \|(\bq',Q') - (\bq,Q)\|_m  
\]
holds for any $(\bq,Q)$ and $(\bq',Q')$ in a 
$\|\cdot\|_m$-neighbourhood of $(-z\phi_0, 0)$, 
where $(\bx',t',Q')$ and $(\bx,t,Q)$ are the inverse images of 
$(\bq',Q')$ and $(\bq,Q)$ respectively. 
By the Ancestor-Descendant relation 
(Theorem \ref{thm:ancestor-descendant}), 
the descendant potentials $\cF^g(\bq)$, $g\ge 1$ are the pull-backs 
of the ancestor potentials $\bar\cF^g_t(\bx)+ \delta_{g,1} F^1(t)$
under the inverse map $(\bq,Q) \mapsto (\bx,t,Q)$. 
Therefore, $\cF^g(\bq)$, $g\ge 1$ are holomorphic 
and uniformly continuous with respect to $\|\cdot\|_m$ 
on a common (i.e.\ independent of $g$) 
$\|\cdot\|_m$-neighbourhood of $(-z\phi_0,0)$ in 
$\cH_+ \times \C^r$. 
Because $\|\cdot\|_m \le C \|\cdot\|_{\infty,m+1}$ 
for some $C>0$, the same holds for the norm $\|\cdot\|_{\infty,m+1}$. 
By Lemma \ref{lem:Frechet-Banach} (2), the monomial Taylor 
expansions of $\cF^g(\bq)$, $g \ge 1$ 
converge uniformly and absolutely 
on a common Banach ball for the norm $\|\cdot\|_{\infty, m+3}$. 
The conclusion follows.  
\qed

\subsection{The Proof of
  Theorem~\ref{thm:convergenceimpliesFrechetconvergenceancestor}}

By Theorem \ref{thm:descendant-convergence}, the total descendant
potential $\cZ_X$ is convergent in the sense of Definition
\ref{def:descendant-converge}.  In particular the genus-zero
non-descendant potential $F^0_X$ is convergent.  In this case the
co-ordinate change \eqref{eq:anc-dec-variables} appearing in the
Ancestor--Descendant relation (Theorem \ref{thm:ancestor-descendant})
is an isomorphism between a neighbourhood of $(\bx,t,Q)=(-\phi_0
z,0,0)$ in $\cH_+ \times \C^{N+1+r}$ and a neighbourhood of
$(\bq,t,Q)=(-\phi_0 z, 0, 0)$ in $\cH_+ \times \C^{N+1+r}$.  (See
equation~\ref{eq:H+} for $\cH_+$.)  This follows from the Nash--Moser
inverse function theorem, using almost the same argument as in the
proof of Theorem~\ref{thm:convergence-genuszero}.  Therefore the
genus-$g$ ancestor potentials $\bar\cF^g_t(\bx)$ for all $g\ge 1$ are
holomorphic on a common neighbourhood of $(\bx,t,Q) = (-\phi_0 z,
0,0)$, and thus are NF-convergent by Lemma \ref{lem:Frechet-Banach}.
(We will prove in the next paragraph that the constants $C$,
$\epsilon$ defining the radius of NF-convergence here can be taken to
be independent of $g$.)\phantom{.}  At genus zero, the
Ancestor--Descendant relation takes the form
\cite{Givental:quantization}*{Proposition~5.3}:
\[
\cF^0(\bq) = \bar{\cF}^0_t(\bx) + \frac{1}{2} W_t(\bq,\bq) 
\]
where the quadratic form $W_t(\bq,\bq)$ appeared in the proof of
Theorem \ref{thm:reconstruction}; it is convergent and holomorphic if
$\bq$ lies in $\cH_+$ and $|t^\alpha|$, $|Q_i|$ are sufficiently small
(cf.~the proof of Theorem \ref{thm:convergence-genuszero}).  The
NF-convergence of $\bar\cF^0_t(\bx)$ follows. 

For the uniformity of the constants $C$, $\epsilon>0$, we use an
argument similar to the proof (\S
\ref{subsec:proof-descendant-convergence}) of
Theorem~\ref{thm:descendant-convergence}.  We know that $\cF_g(\bq)$
is uniformly continuous for $\|\cdot\|_n$ on a common (i.e.\
independent of $g\ge 1$) $\|\cdot\|_n$-neighbourhood of
$(\bq,t,Q)=(-z\phi_0,0,0)$ for some $n\ge 0$ and that the map
$(\bx,t,Q) \mapsto (\bq,t,Q)$ is smooth tame; thus the pull-backs of
$\cF^g(\bq)$ by $(\bx,t,Q) \mapsto (\bq,t,Q)$ is uniformly continuous
for $\|\cdot\|_m$ on a common $\|\cdot\|_m$-neighbourhood of
$(\bx,t,Q) = (-\phi_0 z, 0,0)$ for some $m\ge n$.  Then we apply part
2 of Lemma~\ref{lem:Frechet-Banach} to find that the monomial Taylor
expansion of $\bar\cF^g_t(\bx)$ converges uniformly and absolutely on
a common $\|\cdot\|_{\infty, m+3}$-neighbourhood of $(-\phi_0 z, 0 ,
0)$. \qed

\section{Negative Line Bundles Over Compact Toric Varieties}
\label{sec:local}

We now prove Corollary~\ref{cor:local}.  Let $Y$ be a compact toric
variety and let $X$ be the total space of a direct sum $E =
\bigoplus_{j=1}^{j=r} E_j$ of line bundles $E_j$ over $Y$ such that
$c_1(E_j) \cdot d < 0$ whenever $d$ is the degree of a holomorphic
curve in $Y$.  In what follows we take $r=1$, leaving the proof of the
general case (which is very similar) to the reader.  Let
$\overline{X}$ denote the projectivization $\overline{X} =
\mathbb{P}(E \oplus \C)$, and let $X_\infty \subset \overline{X}$
denote the infinity section.  The inclusion $i:X \to \overline{X}$
induces a map:
\[
i_\star : H_2(X;\Z) \to H_2(\overline{X},\Z)
\]
and $e \in H_2(X,\Z)$ satisfies $e = i_\star(d)$ if and only if $e
\cdot [X_\infty] = 0$.  Thus, since $E \to X$ is negative, any non-constant
holomorphic curve in $\overline{X}$ in the class $i_\star(d)$ lies
entirely inside the zero section of $\overline{X}$.  It follows that:
\begin{align*}
  \corr{a_1 \psi_1^{i_1},\ldots,a_n \psi_n^{i_n}}^{\overline{X}}_{g,n,i_\star(d)}
  =
  \corr{i^\star(a_1) \psi_1^{i_1},\ldots, i^\star(a_n) \psi_n^{i_n}}^{X}_{g,n,d}
&&
d \ne 0
\end{align*}
where the right-hand side is defined as a local Gromov--Witten
invariant \cite{Chiang--Klemm--Yau--Zaslow}, and hence that the total
descendant potential $\cZ_X$ occurs as a subseries of the total
descendant potential $\cZ_{\overline{X}}$.  (Note that the degree-zero
Gromov--Witten invariants of $X$ are not defined, because the relevant
moduli spaces of stable maps are not compact, and hence degree-zero
terms are omitted in the definition of $\cZ_X$.)\phantom{.}
Corollary~\ref{cor:toricflag} implies that the total descendant
potential $\cZ_{\overline{X}}$ is NF-convergent in the sense of
Definition~\ref{def:descendant-converge}, and it follows that $\cZ_X$
is NF-convergent too.  The proof of Corollary~\ref{cor:local} is complete.\qed

\begin{proposition}
  \label{pro:toricflagancestor}
  Let $X$ be a compact toric variety or a complete flag variety.  The
  total ancestor potential $\cA_X$ is convergent in the sense of
  Definition~\ref{def:ancestor-converge}, and is NF-convergent in the
  sense of Definition~\ref{def:Frechet-convergence-ancestor}.
  
\end{proposition}

\begin{proof}
  Combine the argument in the proof of Corollary~\ref{cor:toricflag}
  with Theorem~\ref{thm:maintheoremancestors}.
\end{proof}

\begin{proposition}
  Let $Y$ be a compact toric variety and let $X$ be the total space of
  a direct sum $E = \bigoplus_{j=1}^{j=r} E_j$ of line bundles $E_j$
  over $Y$ such that $c_1(E_j) \cdot d < 0$ whenever $d$ is the degree
  of a holomorphic curve in $Y$. The total ancestor potential $\cA_X$
  is convergent in the sense of
  Definition~\ref{def:ancestor-converge}, and is NF-convergent in the
  sense of Definition~\ref{def:Frechet-convergence-ancestor}.
\end{proposition}

\begin{proof}
  Argue as in the proof of Corollary~\ref{cor:local}, but use
  Proposition~\ref{pro:toricflagancestor} in place of
  Corollary~\ref{cor:toricflag}.
\end{proof}

\end{document}